\documentclass[11pt]{article}
\usepackage{amsfonts, amsmath}
\usepackage{amssymb}
\usepackage{amsthm,amscd}
\usepackage{enumerate}
\usepackage{tikz}
\usepackage{lipsum}
\usepackage{subfig}

\newcommand{\EE}{\mathbb E}
\newcommand{\Z}{\mathbb Z}

\newcommand{\R}{\mathbb R}

\newcommand{\bna}{\begin{eqnarray}}
\newcommand{\ena}{\end{eqnarray}}
\newcommand{\ba}{\begin{eqnarray*}}
\newcommand{\ea}{\end{eqnarray*}}
\newcommand{\bs}[1]{}
\newtheorem{lemma}{Lemma}
\newtheorem{proposition}{Proposition}
\newtheorem{theorem}{Theorem}
\newtheorem{corollary}{Corollary}
\newtheorem{remark}{Remark}
\newtheorem{definition}{Definition}
\newtheorem{conjecture}{Conjecture}

\newcommand{\N}{{\mathcal N}}
\def\p{{\bf p}}

\def\q{{\bf q}}

\usepackage[margin=1in]{geometry}
\begin{document}
\title{Packings of equal disks in a square torus}

\author{Robert Connelly, Matthew Funkhouser, Vivian Kuperberg, 
and 
Evan Solomonides }
\maketitle 

\begin{abstract}  Packings of equal disks in the plane are known to have density at most $\pi/\sqrt{12}$, although this density is never achieved in the square torus, which is what we call the plane modulo the square lattice.  
We find packings of disks in a square torus that we conjecture to be the most dense for certain numbers of packing disks, using continued fractions to approximate $1/\sqrt{3}$ and $2-\sqrt{3}$.   We also define a constant to measure the efficiency of a packing motived by a related constant due to Markov for continued fractions. One idea is to use the unique factorization property of Gaussian integers to prove that there is an upper bound for the Markov constant  for grid-like packings.  By way of contrast, we show that an upper bound by Peter Gruber \cite{Gruber-optimal, Gruber-2-manifolds} for the error for the limiting density of a packing of equal disks in a planar square, which  is on the order of $1/\sqrt{N}$, is the best possible, whereas for our examples for the square torus, the error for the limiting density is on the order of $1/N$, where $N$ is the number of packing disks.


\end{abstract}
\section{Introduction} \label{section:introduction}

It is well-known that the most dense packing of equal disks in the plane is achieved by the hexagonal lattice packing as in Figure \ref{fig:Dense}, which is an old theorem of Thue \cite{Thue}.  See also the simpler proof of L\'aszlo Fejes T\'oth \cite{Fejes-Toth-book} as well as a more recent proof due to  Hai-Chau Chang and Lih-Chung Wang \cite{Chang-Wang}.  Nevertheless, when there are an infinite number of disks in the whole plane, there are an infinite number of non-congruent packings that achieve that maximal density of $\pi/\sqrt{12} = 0.906899\dots$.  Simply remove any finite number of disks from the packing of Figure \ref{fig:Dense}, for example.  

\begin{figure}[h]
    \begin{center}
        \includegraphics[width=0.5\textwidth]{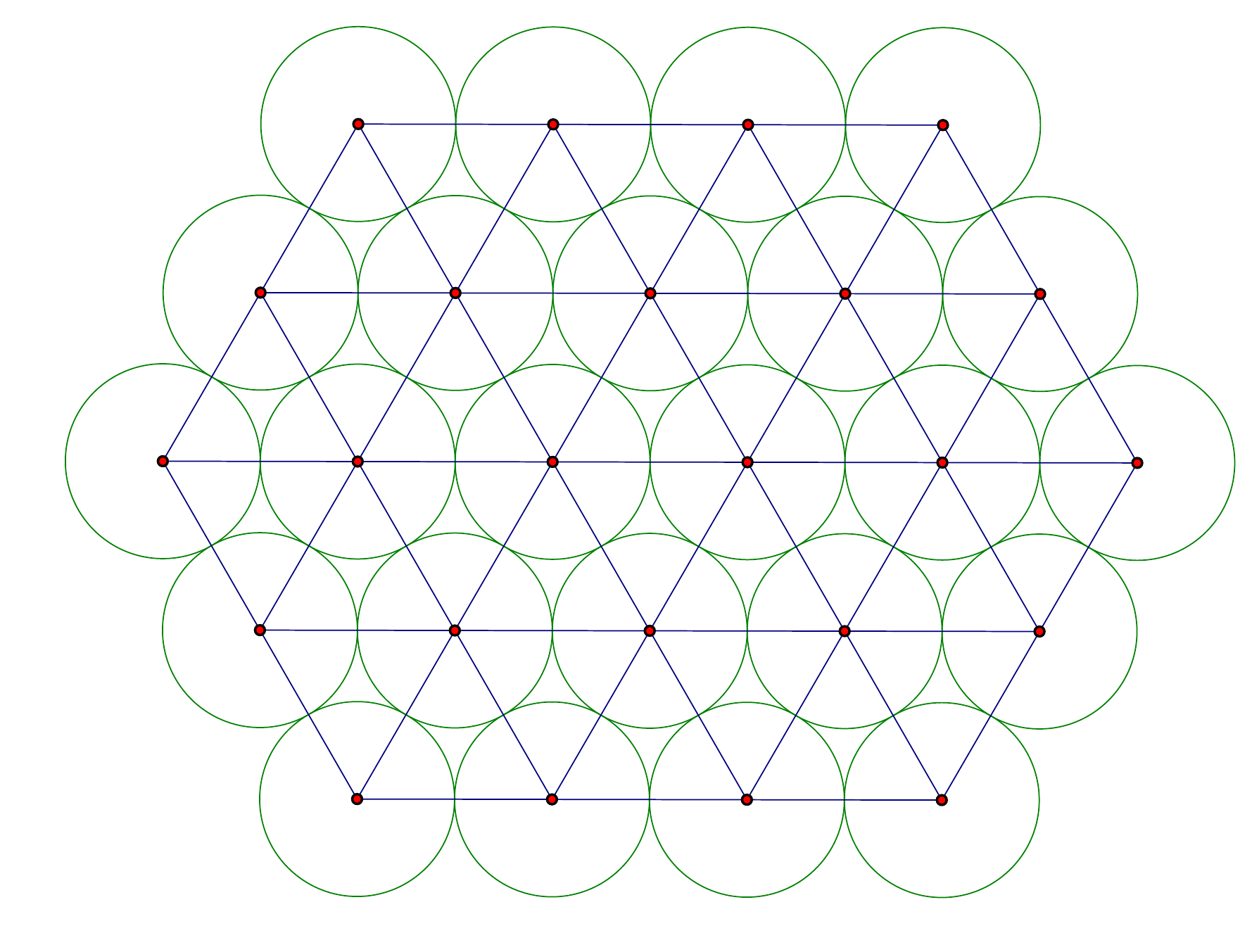}%
        \end{center}
    \caption{}
    \label{fig:Dense}
    \end{figure}

However, when the (equal) packing disks are confined in a compact container such as a torus or a polygon in the plane, the most dense configurations are much less numerous.  There is a lot of literature conjecturing or proving the maximal density and the corresponding configurations achieving that maximal density for a planar square as in Section 1.6 of \cite{Brass-Moser-Pach} and especially the book by Szab\'o et al. \cite{Szabo}.  

In \cite{Connelly-Dickinson}, for the torus given by a triangular lattice, it is shown that if the number of disks is of the form $N=a^2 +ab + b^2$, for $a, b$ integers, then the maximal density $\pi/\sqrt{12}$ can be achieved by embedding a triangular lattice into the the torus as a sublattice. Furthermore if $N$ is NOT of the form $a^2 +ab + b^2$, it is conjectured that the maximal density is at most $(N/(N+1))\pi/\sqrt{12}$.  In other words, in this case, if $\delta_{\Delta}(N)$ is the maximal density for $N$ disks in the triangular lattice torus, the conjecture maintains that the error satisfies
\[
\frac{\pi}{\sqrt{12}}-\delta_{\Delta}(N) \ge \left(\frac{\pi}{\sqrt{12}}\right)\frac{1}{N+1}.
\]

In the case of the torus given by the square lattice (the \emph{square torus}) we show, Theorem \ref{thm:limit}, that the limiting density $\pi/\sqrt{12}$ cannot be achieved, it is always smaller.  Furthermore, we show that there is a constant $K_0$ and there are two infinite sequences of numbers $N$ and corresponding packings with maximum packing density $\delta(N)$ such that 
\[
\frac{\pi}{\sqrt{12}}-\delta(N) \le\frac{1}{K_0 N}.
\]
We also conjecture that there is another constant $K_1$ such that for any packing of the square torus with density $\delta(N)$, the following holds:
\begin{equation}\label{eqn:Markov}
\frac{\pi}{\sqrt{12}}-\delta(N) \ge\frac{1}{K_1 N}.
\end{equation}\label{eqn:Markov}

By way of contrast, for packings of equal disks in a square (or any compact convex set in the plane) with maximum density $\delta_{\Box}(N)$, Theorem 9(i) of Peter Gruber \cite{Gruber-optimal} implies that there is a constant $K_2$ such that  
\[
\frac{\pi}{\sqrt{12}}-\delta_{\Box}(N) \le\frac{1}{K_2\sqrt{N}}.
\]
See \cite{Gruber-2-manifolds} for the original proofs. We give a proof here, Theorem \ref{thm:best-square}, that Gruber's bound for a square is the best possible in the sense that there is another constant $K_3$ such that
\[
\frac{\pi}{\sqrt{12}}-\delta_{\Box}(N) \ge\frac{1}{K_3\sqrt{N}}.
\]
Our proof here uses a theorem of Norman Oler \cite{Oler} (see a shorter proof by Folkman and Graham \cite{Graham}) in Section \ref{section:limiting-density} using a proof of Thue's theorem due to Chang and Wang \cite{Chang-Wang}.  

For $N \le 8$ the most dense packings of equal disks in a square torus are known.  For $N\le 5$ the optimal configurations and their densities where shown by Dickinson, et. al. \cite{Will-square}, and for $N=6,7,8$ by Musin and Nikitenko \cite{Musin}.  We show these packings in Subsection \ref{subsection:Musin}.

Considering the estimates above and motivated from basic results in the theory of continued fractions, we define the \emph{Markov constant for a packing} as the number 
\begin{equation}\label{eqn:Markov1}
M(N)=\frac{\pi}{2}\frac{1}{(\frac{\pi}{\sqrt{12}}-\delta(N))N},
\end{equation}
where $\delta(N)$ is the packing density.
The larger $M(N)$, the ``better" the packing is.  (The constant $\frac{\pi}{2}$ is inserted to so it can be compared with number theory constants.)  We suspect that if $M(N)$ is large enough, independent of $N$, then the packing either is the most dense for that $N$ or it is very close to a maximally dense packing.  If the constant $K_1$ in the inequality (\ref{eqn:Markov}) exists, then there is an upper bound for the Markov constant $M(N)$. In Section \ref{section:Best-Packings} we show the two sequences of packings that have a large Markov constant $M(N)$, and we conjecture that these best packings are maximally dense for the appropriate values of $N$.  Furthermore, we show another pair of sequences of packings that are not maximally dense, but still have such a high Markov constant that any one of them should be quite close to a maximally dense packing.  All of these packings have the property that the graph of their packing configuration, where centers of the packing disks are joined by an edge if the disks touch, is a subgraph of a corresponding triangular lattice.  Indeed, on Page 9 of \cite{Gruber-optimal}, just after Theorem 10, Peter Gruber implies that if the radius of the circles is small enough and the density is very close to $\frac{\pi}{\sqrt{12}}$ the packing is almost regular hexagonal, which may imply the subgraph property above.

One of our basic ideas is to use the theory of continued fractions as in Subsection \ref{section:Continued-Fractions} to find highly dense packings.  This is very closely related to a similar idea for packings in a square by Kari Nurmela, Patric Ostergard, and Rainer ausdem Spring in \cite{Nurmela}.  There are some direct connections as explained in Section \ref{section:square-packings}.

\section{Best or Nearly Best Packings} \label{section:Best-Packings}

\subsection{Rigid Packings} \label{subsection:Rigid-Packings}

Consider a packing of circular disks in the square torus as shown in Figure \ref{fig:example}.  We think of the centers placed at the point $(0,0)$ and thereafter translated by the vectors $(a,b)$ and $(c,d)$ which determine a sublattice of the unit lattice.  The lattice $\Lambda_0(N)$ is chosen so that its generators are $(N,0)$ and $(0,N)$, and it is a sublattice of the $(a,b),(c,d)$ lattice. This is shown in Figure \ref{fig:example}. A fundamental region that determines the square torus is indicated by the dashed square.
\begin{figure}[h]
    \begin{center}
        \includegraphics[width=0.6\textwidth]{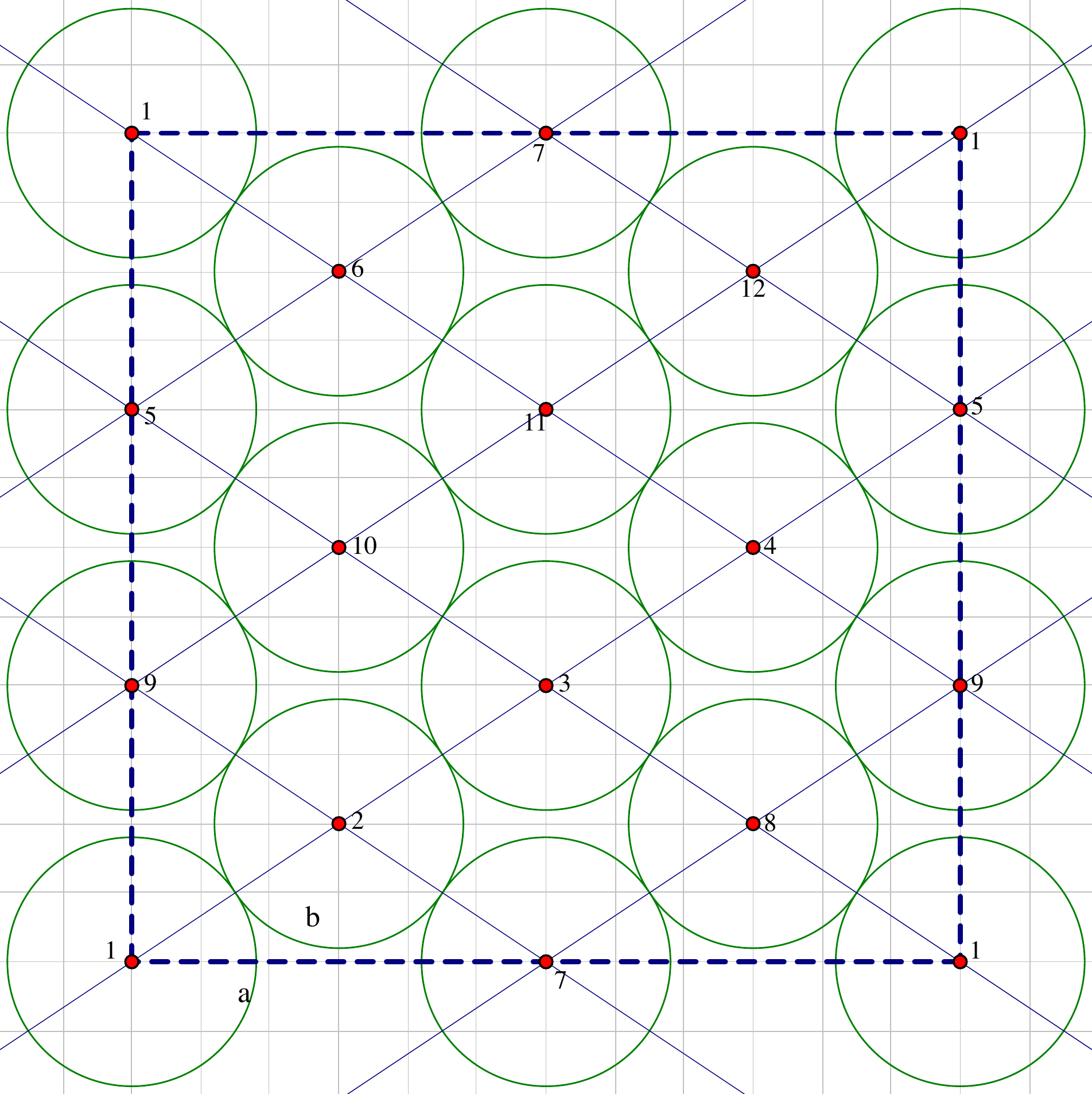}%
        \end{center}
    \caption{This is the square torus packing of $N=12$ disks determined by $(3,2)$ and $(-3,2)$.  The disks are labeled as one proceeds along the geodesic determined by $(3,2)$.}
    \label{fig:example}
    \end{figure}
    
\subsubsection{Type I Packings} \label{subsection:Type-I-Packings}
  
The vectors $(a,b)$ and $(-a,b)$ determine a sublattice $\Lambda_1(a,b)$ of the unit lattice given by $(1,0)$ and $(0,1)$.  A packing at the vertices of this lattice is called a \emph{Type I packing} if the lattice is symmetric about the $x$-axis and symmetric about the $y$-axis.  The index of $\Lambda_0(2ab)$ in the unit lattice is $(2ab)^2$, and the index of $\Lambda_1(a,b)$ in the unit lattice is the determinant:
\[
\begin{vmatrix}
    a       & b \\
    -a       & b \\
\end{vmatrix} =2ab.
\]
So the index of $\Lambda_0(N)$ in $\Lambda_1(a,b)$ is also $N=2ab$, and the number of disks in the square lattice $\Lambda_0(a,b)$ is $2ab$. 

If $b/a > \tan 30^{\circ}=1/\sqrt{3}$, the diameter of the disks can be chosen to be $\sqrt{a^2+b^2}$, the radius of the disks is $\sqrt{a^2+b^2}/2$, and the density of the packing is 
\[
\delta_1(a,b)=\frac{\pi}{4} (a^2+b^2)2ab/(2ab)^2=\frac{\pi}{8} (a^2+b^2)/ab=\frac{\pi}{8}\left(\frac{a}{b}+\frac{b}{a}\right).
\]
So the error is 
\[
\frac{\pi}{2\sqrt{3}}-\delta_1(a,b)=\frac{\pi}{2\sqrt{3}}-\frac{\pi}{8}\left(\frac{a}{b}+\frac{b}{a}\right)=\frac{\pi}{2}\left[\frac{1}{\sqrt{3}}-\frac{1}{4}\left(\frac{a}{b}+\frac{b}{a}\right)\right].
\]
It is natural to see that the closer $b/a$ is to $1/\sqrt{3}$, but greater than $1/\sqrt{3}$, the larger the density.  This is a construction of Type I packings.

We can also handle the case when the $b/a < 1/\sqrt{3}$.  Then the triangle $(0,0), (a,b), (a,-b)$ has shortest side of length $2b$, and so we can insert circles of radius $b$ at each vertex to obtain a packing, although it will not be rigid, (and thus its density can be improved slightly).  In this case its density is 
\[
\delta_1'(a,b)=\pi b^2 2ab/(2ab)^2=\frac{\pi}{2} \frac{b}{a}.
\]
So the error is 
\[
\frac{\pi}{2\sqrt{3}}-\delta'_1(a,b)=\frac{\pi}{2\sqrt{3}}-\frac{\pi}{2}\frac{b}{a}=\frac{\pi}{2}\left[\frac{1}{\sqrt{3}}-\frac{b}{a}\right].
\]

\subsubsection{Type II Packings} \label{subsection:Type-II-Packings}

The vectors $(a,b)$ and $(b,a)$ determine a sublattice $\Lambda_2(a,b)$ of the unit lattice determined as before. A packing at the vertices of this lattice is called a \emph{Type II packing} if the lattice is symmetric about the $x=y$ line and symmetric about the $x=-y$ line.   The index of $\Lambda_0(a,b)$ in the unit lattice is $(a^2-b^2)^2$, and the index of $\Lambda_2(a,b)$ in the unit lattice is the determinant:
\[
\begin{vmatrix}
    a       & b \\
    b       & a \\
\end{vmatrix} =a^2-b^2.
\]
So the index of $\Lambda_0(a,b)$ in $\Lambda_2(a,b)$ is also $a^2-b^2$, and the number of disks in the square lattice $\Lambda_0(a,b)$ is $N=a^2-b^2$. 

If $b/a < \tan 15^{\circ}=2-\sqrt{3}$, the diameter of the disks can be chosen to be $\sqrt{a^2+b^2}$, the radius of the disks is $\sqrt{a^2+b^2}/2$, and the density of the packing is 
\[
\delta_2(a,b)=\frac{\pi}{4} \frac{(a^2+b^2)(a^2-b^2)}{(a^2-b^2)^2}=\frac{\pi}{4} \left(\frac{a^2+b^2}{a^2-b^2}\right).
\]
So the error is 
\[
\frac{\pi}{2\sqrt{3}}-\delta_2(a,b)=\frac{\pi}{2\sqrt{3}}-\frac{\pi}{4} \left(\frac{a^2+b^2}{a^2-b^2}\right)=\frac{\pi}{2}\left[\frac{1}{\sqrt{3}}-\frac{1}{2}\left(\frac{a^2+b^2}{a^2-b^2}\right)\right].
\]
 \begin{figure}[h]
    \begin{center}
        \includegraphics[width=0.6\textwidth]{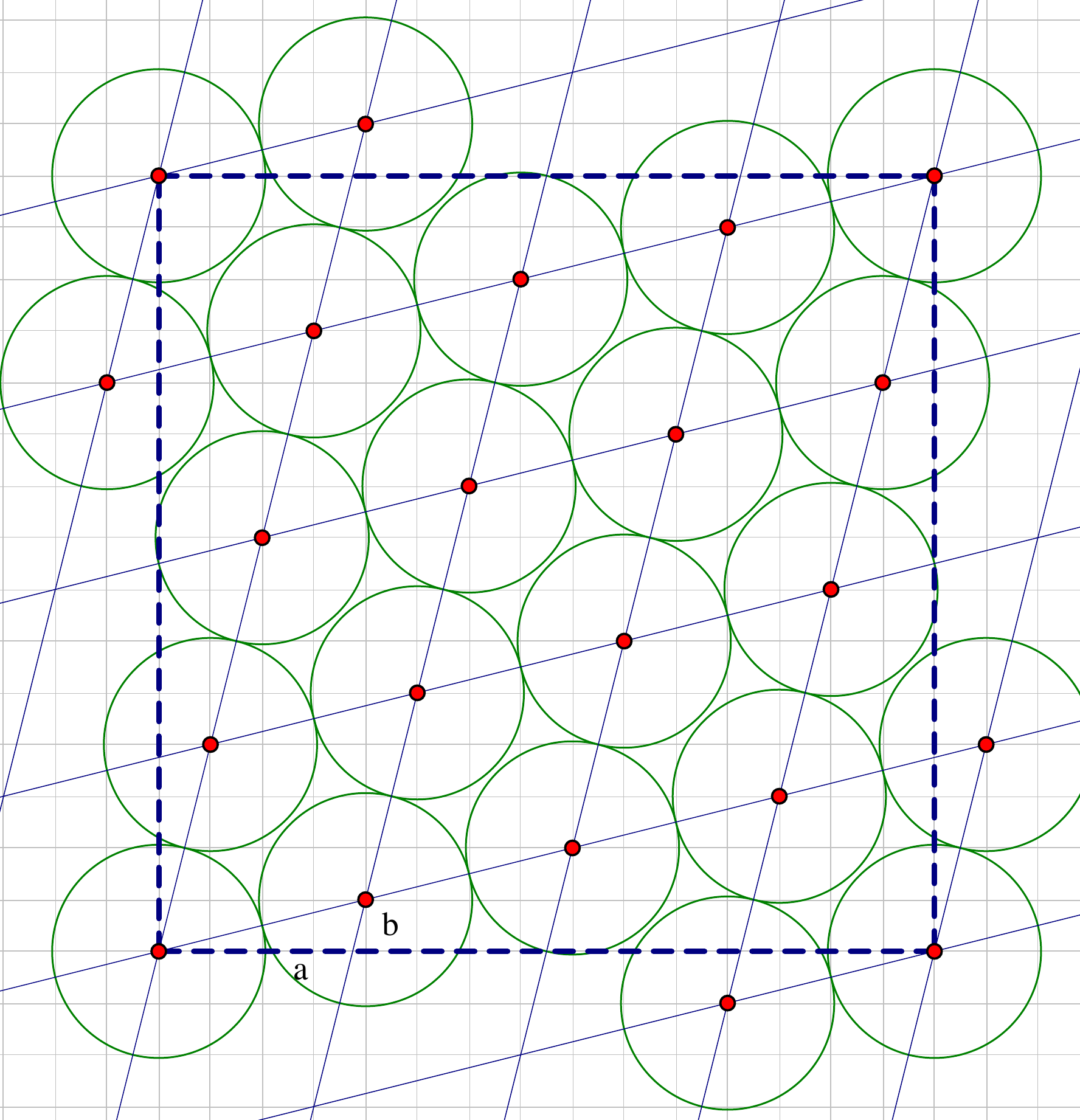}%
        \end{center}
    \caption{This is a rigid Type II square torus packing of $15$ disks determined by the lattice $\Lambda_2(4,1)$.  One can see the gap in the packing disks that are $(-4,4)$ apart.}
    \label{fig:15'}
    \end{figure}
    
These are Type II packings.
Figure \ref{fig:example} is a Type I packing of $12$ disks, and Figure \ref{fig:15'} (as well as Figure \ref{fig:15}) is a Type II packing of $15$ disks.

We can also handle the case when $b/a > \tan 15^{\circ}=2-\sqrt{3}$.  In this case the shortest side of the triangle $(0,0), (a,b), (b,a)$ is the side between $(a,b)$ and $(b,a)$, which has length $\sqrt{2(b-a)^2)}=\sqrt{2}(a-b)$.  So if we place a circle at each vertex with radius $\sqrt{2}(a-b)/2$ we will have a packing, (although it will not be rigid).  The density of this packing is 
\[
\delta'_2(a,b)=\frac{\pi}{2} \frac{(a-b)^2(a^2-b^2)}{(a^2-b^2)^2}=\frac{\pi}{2} \left(\frac{a-b}{a+b}\right).
\] 
So the error is 
\[
\frac{\pi}{2\sqrt{3}}-\delta'_2(a,b)=\frac{\pi}{2\sqrt{3}}-\frac{\pi}{2} \left(\frac{a-b}{a+b}\right)=\frac{\pi}{2}\left[\frac{1}{\sqrt{3}}-\frac{a-b}{a+b}\right].
\]
We will see later in Figure \ref{fig:8} how this packing for $a=3$, $b=1$ can be flexed to a rigid configuration that is known to be maximally dense by \cite{Musin}.
\subsubsection{Rigidity} \label{subsubsection:Rigidity}

Before we go further, we record some of the rigidity information about these sorts of packings.

\begin{proposition}\label{prop:rigid} Let $(a, b)$ and $(c,d)$, $a,b,c,d$ integers, determine a lattice 
$\Lambda_1$ in the plane and let $(ad-bc, 0)$ and $(0,ad-bc)$, $ad-bc \ne 0$, determine a square lattice $\Lambda_0$.  Place struts along the edges corresponding to the generators $(a, b)$ and $(c,d)$.  (In other words, those distances are allowed to increase, but not decrease in distance.)  Then the resulting lattice points form the centers of a rigid packing if and only if the following hold:
\begin{enumerate}[(i.)] 
\item \label{equal-length} The lengths $a^2+b^2 = c^2+d^2$.
\item \label{relatively-prime} The integers $a, b$ are relatively prime and the integers $c, d$ are relatively prime.
\item \label{packing-constraint} The lengths $(a \pm c)^2+(b \pm d)^2 > a^2+b^2$.
\end{enumerate}
\end{proposition}
See \cite{Connelly-Dickinson,Connelly-PackingI,Connelly-PackingII} for a proof. Condition (\ref{packing-constraint}.) is the packing constraint, where the radius is taken to be $\sqrt{a^2+b^2}$.  So this means that, in this case, if one of the conditions of Proposition \ref{prop:rigid} do not hold, then the density of the packing can be strictly increased.

\subsection{Continued Fractions} \label{section:Continued-Fractions}
We see from Subsection \ref{subsection:Rigid-Packings} that we need to find ``good" rational approximations $b/a$ to $1/\sqrt{3}$ and  $2-\sqrt{3}$ for the two types of configurations that we considered. One gets such good approximations using their continued fraction decomposition as follows:  
\[
\frac{1}{\sqrt{3}}= \cfrac{1}{1+\cfrac{1}{1+\cfrac{1}{2+\cfrac{1}{1+\cfrac{1}{2+\cdots}}}}}=[0;1,1,2,1,2,\dots],
\]
where the sequence continues periodically as $1,2$.  Similarly 
\[
2-\sqrt{3}= \cfrac{1}{3+\cfrac{1}{1+\cfrac{1}{2+\cfrac{1}{1+\cfrac{1}{2+\cdots}}}}}=[0;3,1,2,1,2,\dots],
\]
The continued fraction approximations are obtained by truncating the expression above to get a rational approximation to the irrational numbers, $\frac{1}{\sqrt{3}}$ and $2-\sqrt{3}$ in our case, and furthermore the approximations are alternatingly less than and greater than the given irrational.  These convergents are 
\[
\frac{1}{\sqrt{3}}=0.577350269\dots:  0,\, 1,\, \frac{1}{2},\, \frac{3}{5},\,\frac{4}{7},\,\frac{11}{19},\,\frac{15}{26},\,\frac{41}{71},\,\frac{56}{97},\, \frac{153}{265},\,\frac{209}{362},\,\dots.
\]
These fractions then correspond to:
\[
0,\, 2, \,4, \,30, \,56, \,418, \,780,\, 5822, \,10864,\,80190,\,151316, \,\dots,
\]
where the odd cases \dots, $2, 30, 418, 5822, 80190 \dots$ correspond to numbers of Type I rigid packings disks.

Similarly for the convergents for 
\[
2-\sqrt{3}=0.267949192 \dots:  0, \frac{1}{3},\, \frac{1}{4},\, \frac{3}{11},\, \frac{4}{15},\,\frac{11}{41},\,\frac{15}{56},\,\frac{41}{153},\,\frac{56}{209},\, \dots,
\]
these fractions correspond to 
\[
0,\, 8, \,15, \,112, \,209, \,1560, \,2191, \,21728, \,40545, \dots,
\]
where the even cases \dots, 15, 209, 2191, 40545 \dots correspond to numbers of Type II rigid packings.

According to \cite{LeVeque} the \emph{Markov constant} associated with an irrational number $\alpha$ is the quantity $M(\alpha)$ which is the upper limit of those numbers $\lambda$ such that the inequality 
\[
\left| \alpha - \frac{b}{a} \right| < \frac{1}{\lambda a^2}
\]
has infinitely many solutions.
Also according to \cite{LeVeque}:
\begin{theorem}\label{thm:Markov}If $\frac{b}{a}$ is a convergent for $\alpha$ an irrational number, then 
\[
\left|\alpha - \frac{b}{a}\right| < \frac{1}{a^2}.
\]
Furthermore, if 
\[
\left|\alpha - \frac{b}{a}\right| < \frac{1}{2a^2},
\]
then $ \frac{b}{a}$ is a convergent for $\alpha$.
\end{theorem}
For the cases we are interested in, for the convergents $b_k/a_k$ for $\frac{1}{\sqrt{3}}$ and $2-\sqrt{3}$, we can compute the Markov constant exactly for both the limit of the upper sequences and the limit of the lower sequences since the continued fraction coefficients are eventually periodic of period two.  

We first consider the convergents for the $\frac{\sqrt{3}+1}{2}=[1, 2, 1, 2, \dots]$ since it is simply periodic.  The convergents are $1, \frac{3}{2}, \frac{4}{3}, \frac{11}{8}, \frac{15}{11}, \frac{41}{30} \dots$, and if $\frac{p_k}{q_k}$ is the $k$-th convergent, then $\frac{p_{k+2}}{q_{k+2}}=\frac{3p_{k}+q_k}{2p_{k}+q_k}$.  In terms of matrices, 
\[
M \begin{bmatrix}
    p_k        \\
    q_k       \\
\end{bmatrix} =
 \begin{bmatrix}
    3       & 1\\
    2       & 1 
\end{bmatrix} 
\begin{bmatrix}
    p_k        \\
    q_k       
\end{bmatrix} =
\begin{bmatrix}
    p_{k+2}        \\
    q_{k +2}      \\
\end{bmatrix}, 
\]
where  $\frac{p_1}{q_1}=\frac{1}{1}$, and $\frac{p_2}{q_2}=\frac{3}{2}$,  the odd convergents are less that the limit, while the even terms are greater than the limit $\frac{\sqrt{3}+1}{2}$.  

The eigenvalues of the matrix $M$ are $\lambda_1=2+\sqrt{3}$ and  $\lambda_2=2-\sqrt{3}$, and there is a two-by-two matrix $X$ such that for $k= 1, 2, \dots$
\begin{eqnarray}
M^k=
\begin{bmatrix}
    p_{2k}       & p_{2k-1} \\
    q_{2k}       &   q_{2k-1} \\
\end{bmatrix}&=&
X^{-1}
\begin{bmatrix}
    \lambda_1^k       & 0 \\
    0       &   \lambda_2^k \\
\end{bmatrix} 
X  \nonumber \\
&=&\label{eqn:bigsub}
\begin{bmatrix}
    \frac{1}{6} \sqrt{3}(\sqrt{3}+1)     & \frac{1}{6}\sqrt{3} \\
    \frac{1}{3} \sqrt{3}       &  \frac{1}{2} - \frac{1}{6}\sqrt{3}\\
\end{bmatrix}\lambda_1^k +
\begin{bmatrix}
    \frac{1}{6} \sqrt{3}(\sqrt{3}-1)       &  -\frac{1}{6}\sqrt{3} \\
    -\frac{1}{3} \sqrt{3}       &   \frac{1}{2} + \frac{1}{6}\sqrt{3} \\
\end{bmatrix} \lambda_2^k .
\end{eqnarray}
Using (\ref{eqn:bigsub}) one can calculate limiting functions of $p_k$ and $q_k$ using that $\lambda_1^k \rightarrow \infty$, $\lambda_2^k \rightarrow 0$, and $\lambda_1\lambda_2 =1$ as $k\rightarrow \infty$.

Since 
\[
\frac{1}{\sqrt{3}} = \frac{1}{2\left(\frac{\sqrt{3}+1}{2}\right)-1},
\]
the sequence 
\[
\frac{b_k}{a_k}= \frac{1}{2\left(\frac{p_k}{q_k}\right)-1}=\frac{q_k}{2p_k-q_k}
\]
converges to $\frac{1}{\sqrt{3}}$, while $b_k=q_k$ and $a_k=2p_k-q_k$ are the convergent numerator and denominator for $\frac{1}{\sqrt{3}}$ for $k$ odd, and $b_k=q_k/2$ and $a_k=(2p_k-q_k)/2$ are the convergent numerator and denominator for $\frac{1}{\sqrt{3}}$ for $k$ even.

Similarly, since
\[
2-{\sqrt{3}} = -2\left(\frac{\sqrt{3}+1}{2}\right)+3,
\] 
the sequence 
\[
\frac{b_{k-1}}{a_{k-1}}=-2\left(\frac{p_k}{q_k}\right)+3=\frac{-2p_k+3q_k}{q_k}
\]
converges to $2-{\sqrt{3}}$, for $k=1, 2, \dots$, where $b_{k-1}=-2p_k+3q_k$ and $a_{k-1}=q_k$ are the convergent numerator and denominator for $2-{\sqrt{3}}$ for $k$ odd, and $b_{k-1}=(-2p_k+3q_k)/2$ and $a_{k-1}=q_k/2$ are the convergent numerator and denominator for $2-{\sqrt{3}}$ for $k$ even.

\subsection{Markov Constants} \label{subsection:Markov}
We will extend the definition of a Markov constant to apply to planar packings instead of rational approximations to irrational numbers.  The maximal density of a planar packings of equal circles is $\frac{\pi}{2\sqrt{3}}$.  If we have a packing of $N$ equal circles in a (square) torus with density $\delta = \frac{\pi}{2}\bar{\delta}$, we define the \emph{Markov constant} of that packing to be 
\[
M=M(N,\delta)=\frac{1}{\left(\frac{1}{\sqrt{3}}-\bar{\delta}\right)N}.
\]
The reason for the constant $\frac{2}{\pi}$ is so we can compare the density calculations with the number theory terms in Theorem \ref{thm:Markov}.  See the same definition from the introduction (\ref{eqn:Markov1}).  Table \ref{table:1} shows the limit for Type I packings (on the left) Type II packings (on the right) for both the rigid and non-rigid cases in terms of the convergents to $\frac{1}{\sqrt{3}}=\tan 30^{\circ}$ and $2-\sqrt{3}=\tan 15^{\circ}$, respectively.

\begin{table}[h!]
  \centering
  \label{tab:table1}
  \begin{tabular}{|c||c|c||c|c|}
  \hline
  	&   & Type I & &  Type II\\
     \hline
    $M^{-1}$ 
    & $(\frac{1}{\sqrt{3}}-\frac{b_k}{a_k})a_k^2$ 
    & $\left[\frac{1}{\sqrt{3}}-\frac{1}{4}\left(\frac{a_k}{b_k}+\frac{b_k}{a_k}\right)\right](2a_kb_k)$ 
    & $(2-{\sqrt{3}}-\frac{b_k}{a_k})a_k^2$ 
    &$\left[\frac{1}{\sqrt{3}}-\frac{a_k-b_k}{a_k+b_k}\right]\left(a_k^2-b_k^2\right)$\\
    \hline
    \hline
   odd $k$ & $-\sqrt{3}$ & 3 &$-\sqrt{3}$&$3/2$ (not rigid)\\
    \hline
  even $k$ &$2\sqrt{3}$ & 3 (not rigid) &$2\sqrt{3}$ & $6$\\
    \hline
    \hline
$M^{-1}$ & &$(\frac{1}{\sqrt{3}}-\frac{b_k}{a_k})2a_kb_k$  & & $ \left[\frac{1}{\sqrt{3}}-\frac{1}{2}\left(\frac{a_k^2+b_k^2}{a_k^2-b_k^2}\right)\right](a_k^2-b_k^2)$\\
\hline
 \end{tabular}
   \caption{Markov constants.}\label{table:1}
\end{table}
Figure \ref{fig:Dense-plot} shows some of the generating points on a square grid for various packings corresponding to some of the packings in Table \ref{table:1}.  
\begin{figure}[h]
    \begin{center}
        \includegraphics[width=1.0\textwidth]{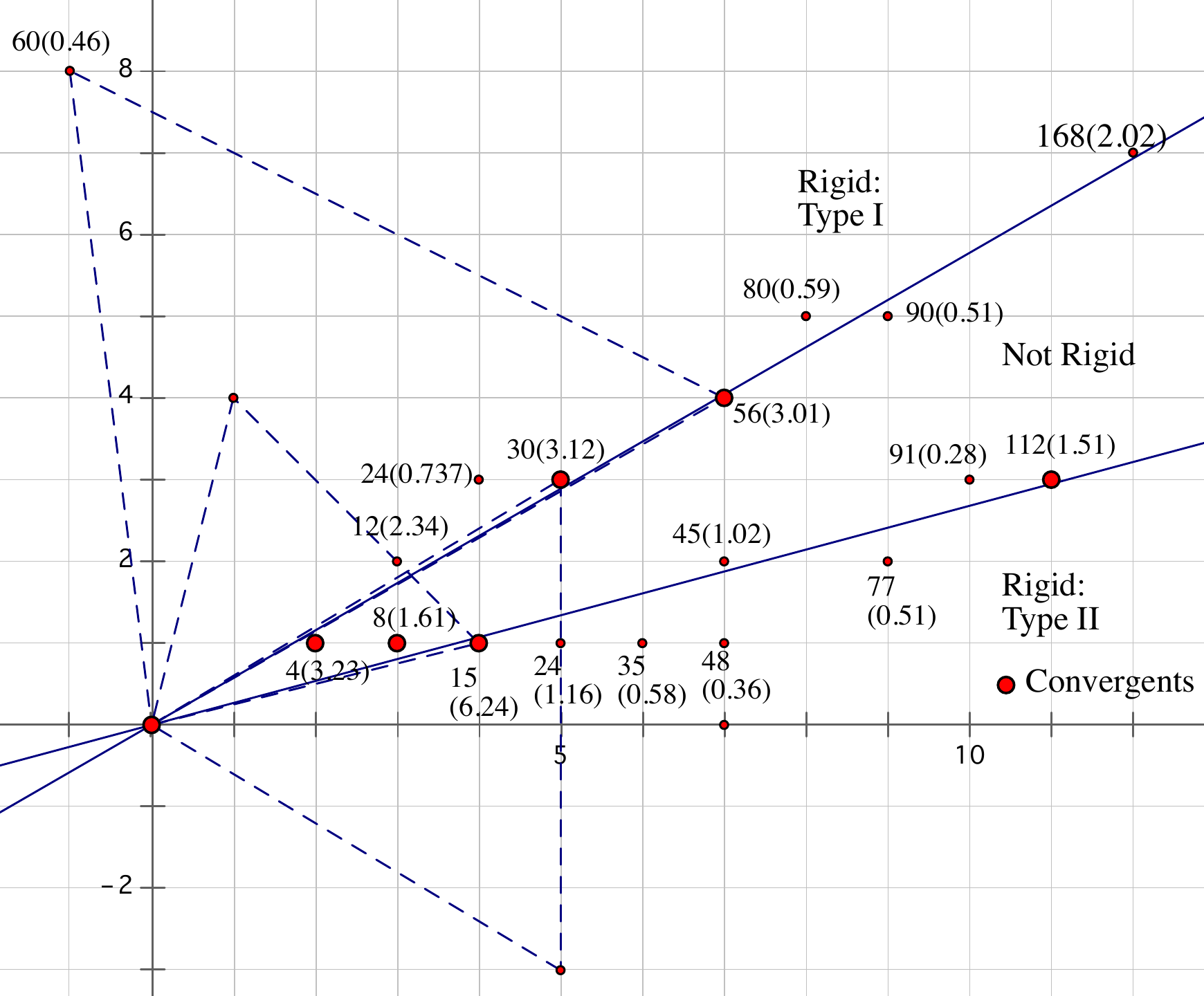}%
        \end{center}
    \caption{The number of the packing disks is indicated along with the corresponding Markov constants in parenthesis.  The upper left dashed triangle corresponds to a rigid packing that is neither of Type I nor Type II, where two lengths of its edges are the same since $8^2 +1^2=7^2+4^2=65$.  Note that there are two rigid non-convergent grid-like packings of $24$ disks of different densities and Markov constants, one of Type I and the other of Type II.}
    \label{fig:Dense-plot}
    \end{figure}
\clearpage

\section{Other Grid-like Packings} \label{section:Grid-like}
We say that a packing is \emph{grid-like} if the graph of the packing consists of two geodesics, each of which contains all the packing centers.  Type I and Type II packings are both grid-like, and we will relate all of the others as ``lifts" of Type I and Type II packings, which we will explain.  We will identify all the elements in the unit square lattice as the standard Gaussian integers $\Z[i]$, and we will use the fact that $\Z[i]$ is a principle ideal domain with unique factorization, as explained in \cite{LeVeque}.  Indeed, the primes in $\Z[i]$ are the real integers $p \equiv 3, \mod{4}$; $a \pm bi$, where $a^2+b^2$ is a real prime $p \equiv 1, \mod{4}$;  and $1 \pm i$, as well as their associates obtained by multiplying them by a unit $\pm i$ and $-1$.  Following standard notation, we define the \emph{norm} of $a + bi \in \Z[i]$ as $\N(a+bi)=a^2+b^2=(a + bi)\overline{(a + bi)}=(a + bi)(a - bi)$.

If $z_1= a + bi$ and $z_2= c + di$ are generators for a grid-like packing, then the conditions in  Proposition \ref{prop:rigid} imply $\N(z_1)=\N(z_2)$, and both $z_1$ and $z_2$ have only $\pm 1$ as a common real factor.  

\begin{proposition}\label{prop:type-grid} If $z_1= a + bi$ and $z_2= c + di$ are generators for a grid-like rigid packing satisfying the conditions of Proposition \ref{prop:rigid}, then they form:
\begin{enumerate}[(1)]
\item \label{Type-I} a Type I packing if and only if $z_1$ and $z_2$ are relatively prime in $\Z[i]$, except possibly with the common factor $1+i$, and $z_1 z_2$ is real; 
\item \label{Type-II} a Type II packing if and only if $z_1$ and $z_2$ are relatively prime in $\Z[i]$ and $z_1 z_2$ is pure imaginary.
\end{enumerate}
\end{proposition}
\begin{proof}
For Part (\ref{Type-I}), if $z_1= a + bi$ and $z_2= a - bi$ are generators for a Type I packing, $z_2= a - bi=\overline{a+ bi}=\bar{z_1}$, and so in $\Z[i]$, all the prime divisors of $z_2$ are conjugates of prime divisors of $z_1$ and conversely.  If $p \in  \Z[i]$ is a common prime divisor of $z_1$ and $z_2$, then $\bar{p}$ is also a common divisor and except for $1+i$, the real prime $p\bar{p}$ would also be a common divisor contradicting Condition (\ref{relatively-prime}) of Proposition \ref{prop:rigid}.  It is clear that $z_1 z_2=z_1 \bar{z_1}=\N(z_1)$ is real.  

Conversely, if $z_1$ and $z_2$ have no common prime factors in $\Z[i]$, except $1+i$, since $z_1 \bar{z_1}=\N(z_1)=\N(z_2)=z_2 \bar{z_2}$, each prime divisor $p$ of $z_1$  has its conjugate $\bar{p}$ as a divisor of $z_2$, and thus $z_1 z_2$ is a unit times $\N(z_1)=\N(z_2)$ which if the unit is real has to be $\pm1$.  Replacing say $z_2$ by $-z_2$ does not change the packing, so $z_1$ and $z_2$ generate a Type I packing.

For Type II packings $z_1= a + bi$ and $z_2= b + ai$ and $z_1 z_2=(b^2-a^2)i$, and a similar argument shows Part (\ref{Type-II}). $\square$
\end{proof}

For example, $z_1 = (4+i)(1+i)= 3+5i $, $z_2= (1 +4i)(1+i) = -3+5i$ generate a  Type I packing although its factors are not relatively prime. 

There are several examples of rigid grid-like packings other than Type I and Type II packings.  One of the simplest is for a square grid. 

 \begin{proposition}\label{prop:square-grid} Every rigid square grid in the square torus is generated by $z_1$ and $z_2=iz_1$, where $z_1$ is the product of powers of irreducible non-real primes in $\Z[i]$ such that if a prime appears, its conjugate does not.   
 \end{proposition}
 \begin{proof} Note that the prime $1+i$ can appear only to the first order, since its conjugate $1-i=(1+i)(-i)$ is an associate.   Since the grid is a square grid, we must have $z_2=\pm iz_1$, and  without loss of generality we can assume $z_2=iz_1$.  If a prime appears with its conjugate, their product has a positive real integer, greater than one, in the product and this violates Condition (ii.) of Proposition \ref{prop:rigid}.  Otherwise, by unique factorization for $\Z[i]$, the only time a Gausian prime divides its conjugate is when the prime is $\pm 1\pm i$.  Thus Condition (ii.) of Proposition \ref{prop:rigid} holds and the configuration is rigid. $\square$
 \end{proof}
 
Note that a square grid is known to be the most dense packing for $N=1, 2, 5$, and we conjecture it is for $N=10$.  The corresponding sequence of packing numbers is $1, 2,5, 10, 13, 17, 25, 26, 29, \dots$, which is sequence $A008784$ in N. J. A. Sloane's On-Line Encyclopedia of Integer Sequences.  The density for each of these packings is $\pi/4$, and so the Markov constant is $1/N(\frac{1}{\sqrt{3}}-\frac{1}{2})\approx 12.9/N$.  For larger $N$ it seems unlikely that the most dense packing will be a square grid.

 \begin{proposition}\label{prop:lifting}  Let $z_1$ and $z_2$ generate a Type I or Type II rigid packing, and let $w \in \Z[i]$ be such that it has a common (prime) factor with neither $z_1$ nor $z_2$, and has no real prime factor.  Then $z_1 w$ and $z_2 w$ correspond to another rigid grid-like packing coming from an index $\N(w)$ covering of the $z_1, z_2$ packing.  Furthermore, every rigid grid-like packing is obtained by this process corresponding to a Type I, Type II, or square grid packing.
\end{proposition}
\begin{proof} Given $z_1, z_2$ as above, by Proposition \ref{prop:type-grid}, they have no common factors in $\Z[i]$, and $\N(z_1)=\N(z_2)$.  So $\N(z_1w)=\N(z_1)\N(w)= \N(z_2)\N(w)=\N(z_2w)$, and it is easy to see that if there were a common real prime factor in the real and the imaginary parts of $z_1w$ or $z_1w$ it would appear in their $\Z[i]$ factorings.  So  $z_1w$ and $z_1w$ are rigid packing generators with $\N(w)$ times as many packing disks with exactly the same density, since $w$ just rotates $z_1, z_2$ rigidly and then rescales each line by  $\sqrt{\N(w)}$.  

Conversely, given any $z_1, z_2$ that generates a grid-like packing, we can take the common divisors as the Type I, Type II or the square grid packing generators, as in Proposition \ref{prop:square-grid} and treat the common factors as $w$.  $\square$
\end{proof}

Notice that the Markov constant decreases by a factor of $\N(w)$ in the lifting of Proposition \ref{prop:lifting}, and the Type II packing that corresponds to convergents for $2-\sqrt{3}$ converge themselves to $6$ from above.  This leads to the following:
\begin{corollary}\label{cor:II-bound} If $ a +bi, b+ai$, where $a \ge b \ge 1$, generates a Type II packing with Markov constant $M_P>3.59$, then $b/a$ is an even convergent  to $2-\sqrt{3}$ and $6.25 > M_P > 6$.  Therefore any grid-like Type II packing has Markov constant $M_P < 6.25$.
\end{corollary}
\begin{proof} From Table \ref{table:1} define the number theory Markov constant $M_{\#}$ for a Type II packing generated by $ (a, b)$ by
\[
M_{\#}^{-1}=\left(\sqrt{3}-2-\frac{b}{a}\right)a^2, 
\]
and the packing Markov constant $M_{P}$ by
\[
M_{P}^{-1}= \left[\frac{1}{\sqrt{3}}-\frac{1}{2}\left(\frac{a^2+b^2}{a^2-b^2}\right)\right](a^2-b^2). 
\]
Then it turns out that 
\[
M_{\#}=\frac{1}{2\sqrt{3}}\left(\frac{b}{a}\left(2+\sqrt{3}\right)+1\right)M_{P}.
\]
The smallest value for $a$ in a Type II packing is $a=4$ when $b=1$.  From the definition of $M_P$, for fixed $a$, it decreases as $b$ increases, and it decreases for a fixed ratio $\frac{b}{a}$ as $a$ and $b$ increase, subject to $\frac{b}{a} < \sqrt{3} -2$.  Indeed, if $a=6$, and $b= \frac{6}{4}$, then $M_{P}=2.77\dots < 3.59$. (The point $(5,1)$ has $M_{P}=1.16 < 3.59$ as well.  So we can assume that $\frac{b}{a} \ge \frac{1}{4}$, and then 
\[
M_{\#} \ge \frac{1}{2\sqrt{3}}\left(\frac{1}{4}\left(2+\sqrt{3}\right)+1\right)(3.59)
 = 2.00323> 2.
\]
Thus by Theorem \ref{thm:Markov}, $\frac{b}{a}$ is a convergent to $2-\sqrt{3}$, and from the calculations in Table \ref{table:1}, since the even convergents converge to $6$ from above, starting from a bit less than $6.25$, the conclusion follows. $\Box$
\end{proof}

There is a strong connection between Type I and Type II packings, especially those that correspond to convergents.  For example, the complex factors corresponding to Type II rigid convergent packings are $4+i, 15+4i, 56+15i, 209+56i, \dots$, and $\overline{(4+i)}(1+i)=5+3i, \overline{(15+4i)}(1+i)=19+11i, \overline{(56+15i)}(1+i)=71+41i, \overline{(209+56i)}(1+i)=265+153i, \dots$ are the corresponding complex numbers in $\Z[i]$ corresponding to Type I  packings.  Notice that these Type I packings all have the prime factor 1+i in both of the generators, so the generators are rescaled each by $\sqrt{2}$, while the area of the corresponding torus and the number of packing disks is multiplied by $2$.    So the Markov constant is halved, since the density of these corresponding packing is the same.  The first example of this is for the Type II packing in Figure \ref{fig:15'}, having $15$ disks,  which corresponds to the Type I packing in Figure \ref{fig:30'}, having $30$ disks.

\begin{remark}\label{remark:Type-II-Markov} We can do an analysis for Type I packings similar to what was done in Corollary \ref{cor:II-bound} for Type II packings.  Namely, all the Markov constants $M_P$ for Type I packings are less than $3.54$.  We can assume that $\frac{b}{a} \le \frac{3}{5}$ which is the first Type I convergent.  Then for the Markov constant for Type I packings,
\[
M_{\#} = \frac{1}{2}\left(\sqrt{3}-\frac{b}{a}\right)M_P\ge \frac{1}{2}\left(\sqrt{3}-\frac{3}{5}\right)(3.54)=2.004 >2.
\]
So $\frac{b}{a}$ must be an odd covergent to $\frac{1}{\sqrt{3}}$ and $3< M_P < 3.12$.
\end{remark}

In general, if $a+bi = z_1, b+ai=z_2$, corresponds to a Type II packing, then $z_2=i\bar{z_1}$, and $z_1(1-i)=(a+bi)(1-i)=a+b+(b-a)i=w_1$ and $z_2(1-i)=(b+ai)(1-i)=a+b+(a-b)i=w_2$, is a Type I packing, since $\bar{z_1}=b_2$, with exactly the same density.  Similarly if $z_1, z_2=\bar{z_1}$ are generators of a Type I packing then  $z_1(1+i)=w_1, z_2(1+i)=w_2$ will be generators of a Type II packing with the same density and twice the number of packing elements, unless $1+i$ is a common divisor of $z_1, z_2=\bar{z_1}$, in which case after a common factor of 2 is removed $w_1, w_2$ correspond to a Type II packing. (Note that if $1+i$ divides $z_1$, and it divides $z_2=\bar{z_1}$, then $(1+i)^2=2i$ or $1-i$ divides $z_2$, and $(1+i)(1-i)=2$. In both cases the factor of $2$ appears.)

With all this information, we have a complete description of rigid grid-like packings together with their Markov constants.

\begin{theorem} If $z_1$ and $z_2$ are the generators of a rigid grid-like packing with more than two disks, then its Markov constant is $M_P < 6.25$ and $M_P > 3.59$ implies the packing is a Type II packing with $6.25 > M_P > 6$.
\end{theorem}
\begin{proof}  Since $\N(z_1)=z_1 \bar{z_1}=z_2\bar{z_2}= \N(z_2)$, each prime factor $p$ of $z_1$ must divide either $z_2$ or $\bar{z_2}$.  If any such $p$ divides $z_2$, we can factor it out of both sides to get another packing with fewer packing disks having the same density, and thus a larger Markov constant.  Then we end up with a packing with generators such that $z_1$ and $\bar{z_2}$ are associates.  That is each is a unit times the other.  If the unit is real, i.e. $\pm 1$, we can replace   $z_1$ by $\pm z_1$,  and replace $z_2$ by $\pm z_2$  to get other generators of the same packing such that $z_2=\bar{z_2}$ and we have a Type I packing, which we know has a Markov constant $M_P < 3.54$ from Remark \ref{remark:Type-II-Markov}. (If it corresponds to a convergent Type I packing as described above, then $3<M_P < 3.13$).

Otherwise, $z_2=\pm i\bar{z_1}$.  Again we can assume that $z_2= i\bar{z_1}$, and it corresponds to a Type II packing, and from Corollary \ref{cor:II-bound}, the conclusion follows.
$\Box$ \end{proof}

An example of this process is the packing implied by the generators in the upper left of Figure \ref{fig:Dense-plot}, where $z_1=7+4i=(1+2i)(3-2i)$ and $z_2=-1+8i=(1+2i)(3+2i)$.  So they have a common factor of $(1+2i)$, with norm $\N(1+2i)=5$.  The corresponding Type I packing is generated by $3+2i$ and $3-2i$ which has Markov constant $M_P=2.34$ with $12$ disks, while the original packing has $M_P=2.34/5=0.46$ and $5\cdot 12=60$ disks.

\section{Flexing Non-rigid Packings} \label{section:Flexing}

For those Type I and Type II packings that are not rigid, it is possible to increase the packing density, and one hopes to ``flex" it until it is at least rigid or perhaps even maximally dense.  In general, just because the packing is not rigid does not mean that the density can be increased.  For example, there is a maximally dense packing of $7$ disks (shown in Figure \ref{fig:7-1}) that has a ``rattler" that is free from the rest of the packing disks according to \cite{Musin}.  Nevertheless, for Type I and Type II non-rigid packings we can always increase the density, but it can be difficult to see how to do the motion to get rigidity for large values of $N$.  

For $8$ disks we show a sequence of packings indicating the flex as it converges to the maximally dense packing starting with Figure \ref{fig:8}. 

\begin{figure}[h]
    \begin{center}
        \includegraphics[width=0.3\textwidth]{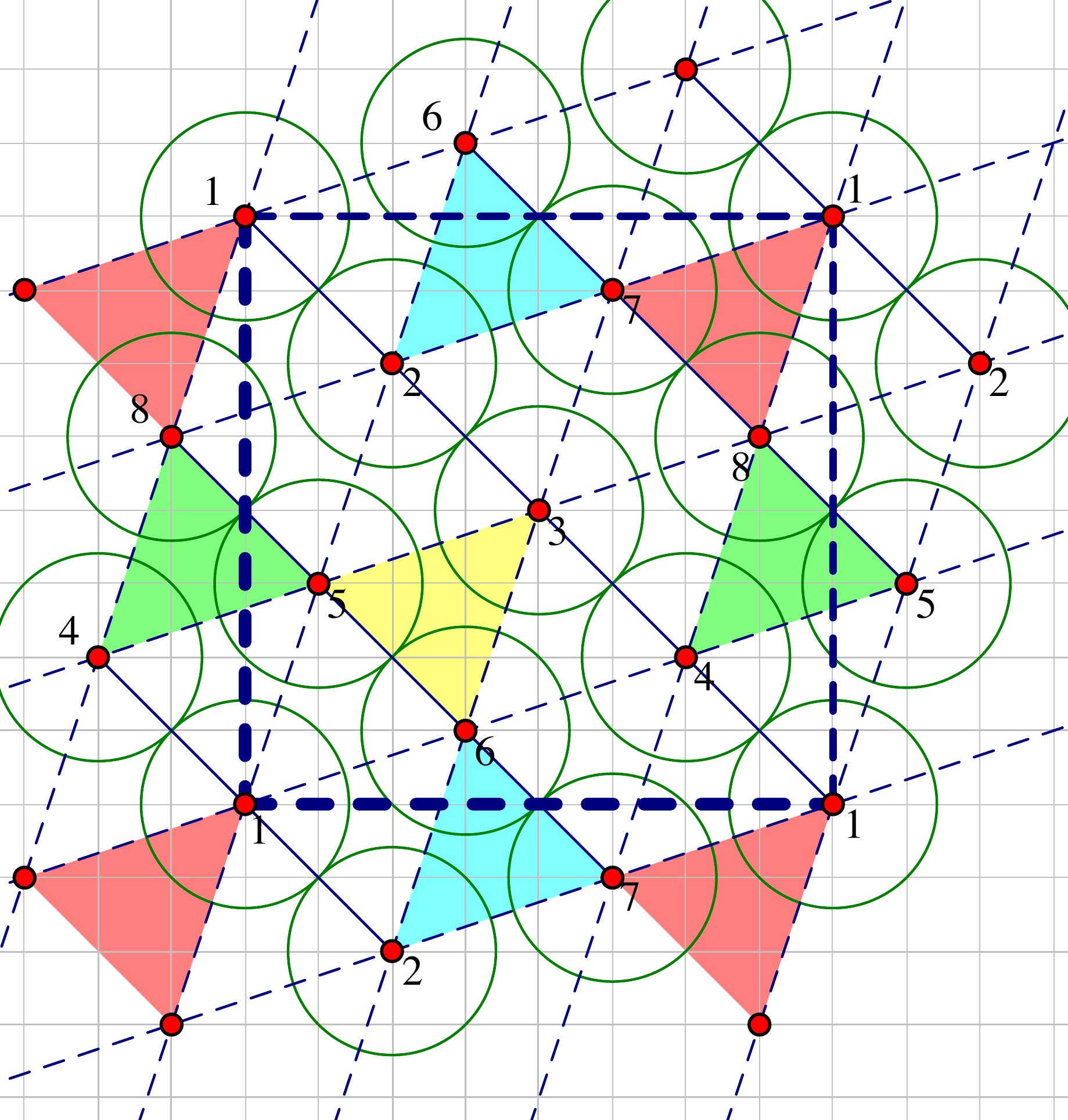}%
        \end{center}
    \caption{This is the initial configuration where the lightly dashed line segments are longer than the minimal length, as one can see from the indicated circles. Notice that the graph of the packing has two components, the centers $1,2,3,4$ and $5,6,7,8$.  The heavily dashed square is a fundamental region.}
    \label{fig:8}
    \end{figure}
\begin{figure}[h]
    \begin{center}
        \includegraphics[width=0.3\textwidth]{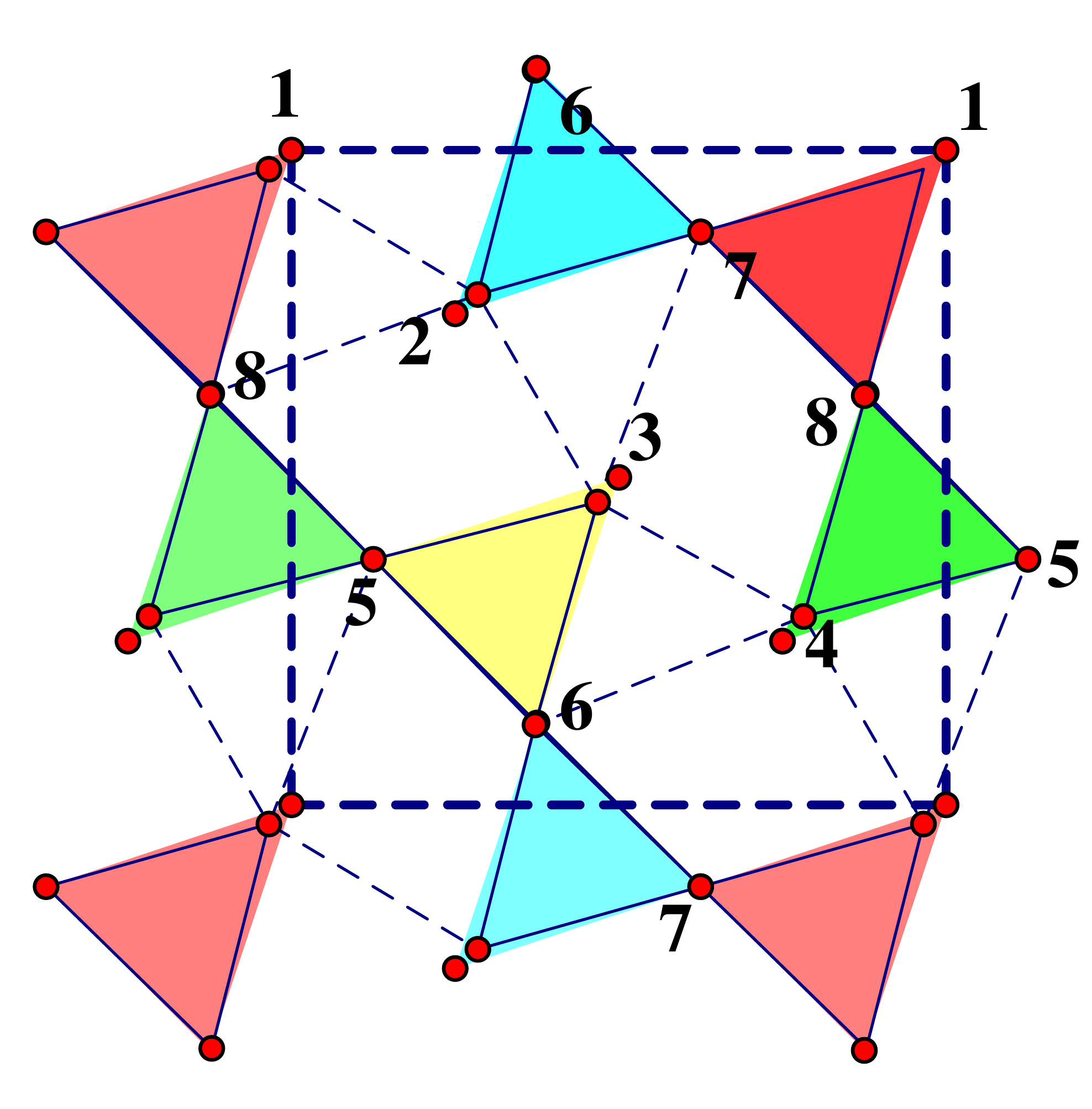}%
        \end{center}
    \caption{Next the packing centers $1,2,3,4$ are moved nearer to the other vertices of each colored triangle until all the colored triangles are equilateral as above.}\label{fig:8-1}
    \end{figure}  
\clearpage      
\begin{figure}[h]
    \centering
    \subfloat[]{{\includegraphics[width=0.3\textwidth]{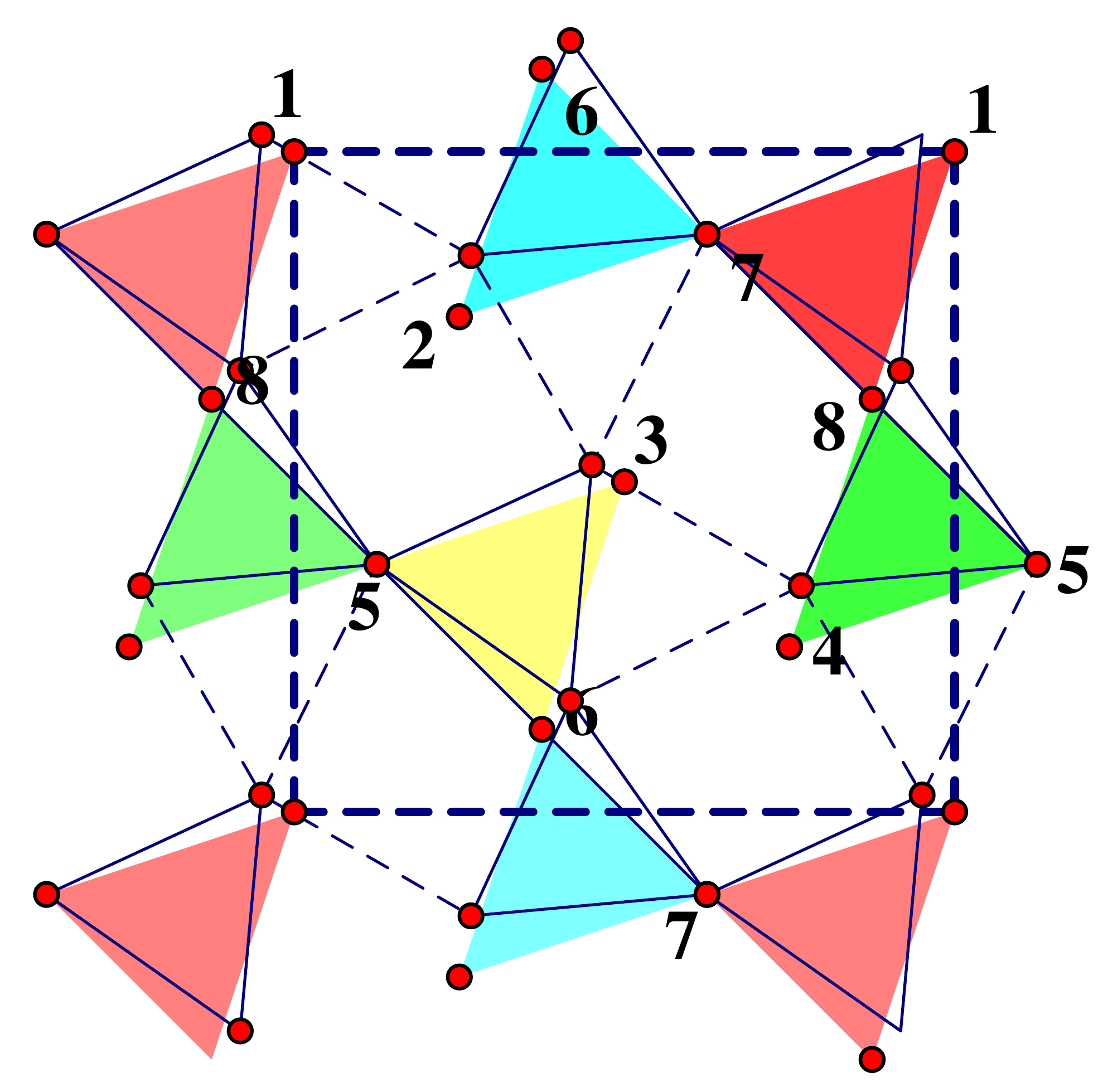} }}%
    \qquad
    \subfloat[]{{\includegraphics[width=0.3\textwidth]{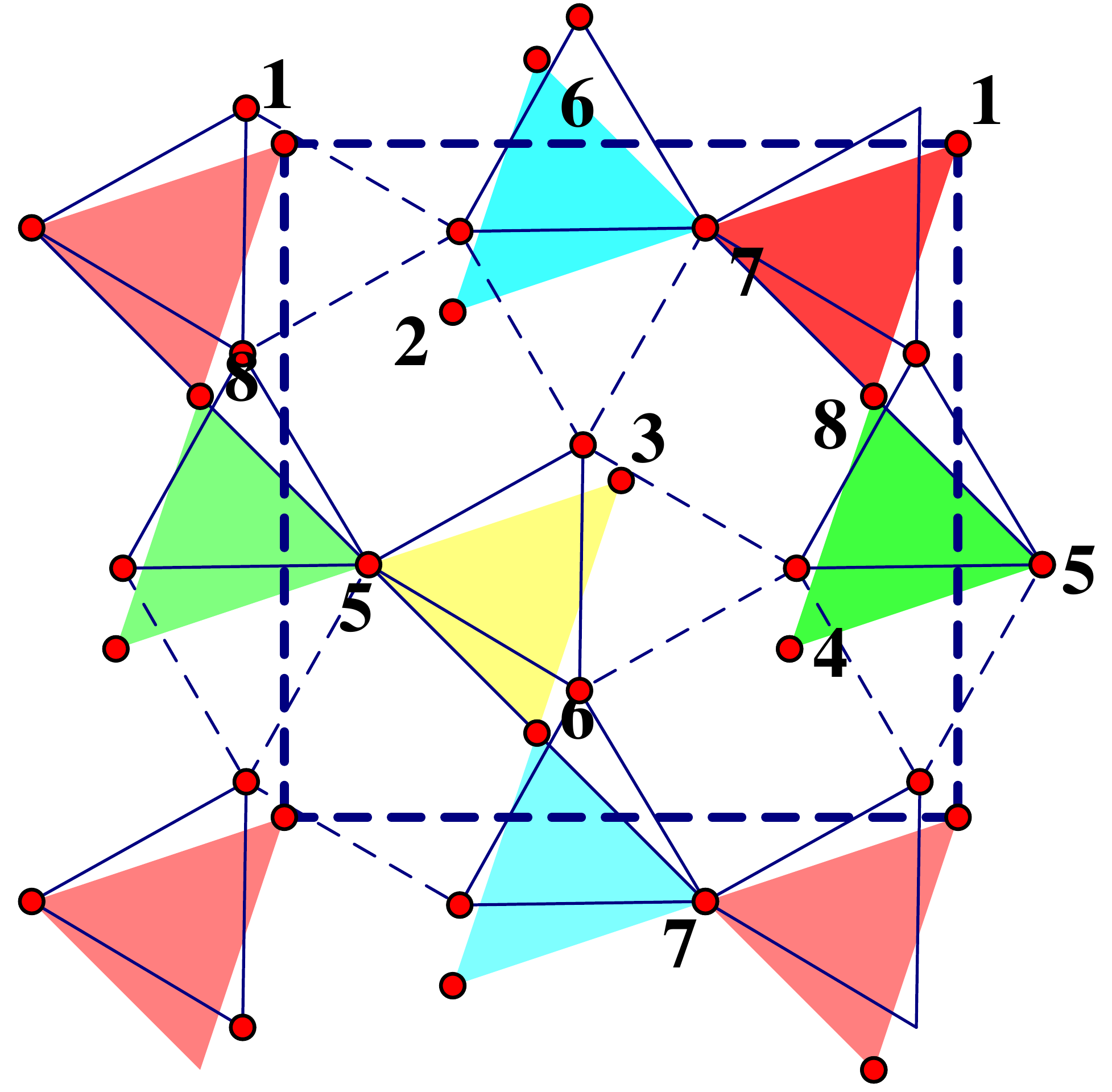} }}%
    \caption{Then fixing centers $5,7$, flex each of the formally colored triangles, keeping them equilateral until the dashed lengths are equal to the lengths of colored triangles as in Figure \ref{fig:8-3}.
This is the final motion.}%
    \label{fig:8-3}%
\end{figure}
 
According to \cite{Musin} the configuration of Figure \ref{fig:8-3}(b) corresponds to the most dense packing of $8$ circles in a square torus.  This most dense packing has Markov constant greater than $3.01$, while the Type II non-rigid packing has Markov constant less that $1.62$.  See Section \ref{section:Best}.

In \cite{Gruber-optimal}, just after his Theorem 10, while discussing `stability' results in geometry, Peter Gruber says about packings and coverings in compact domains 

``In other words: if the radius of the circles of the covering is sufficiently small and the density very close to $2\pi/3\sqrt{3}$, then the covering is `almost regular hexagonal'. A corresponding result for the packing case can be proved in a similar way."  

For the case of packings of equal circles in a square torus, this leads us to the following:

\begin{conjecture}  If a packing of $N$ equal circles in a square torus corresponds to one of the convergents of a $1/\sqrt{3}$ or $2-\sqrt{3}$ of a Type I or Type II packing, respectively, then the packing graph is a subgraph of a triangulation of the torus, where each vertex has degree six.
\end{conjecture}
This seems to be what is happening for the conjectured and proven examples that have been calculated.  If this is proven, this could also help reduce the number of possible examples of graphs that need to be checked in the proofs involved in finding the densest packings.  For the grid-like graphs of Type I and Type II, the triangulation is clear, but it would be interesting to find the additional triangulating edges of a graph of a rigid configuration.  This could, of course, apply to the rigid subgraph of a most dense packing or even include rattlers as with one of the best packing for $7$ disks in a square torus.

\section{Relation to Square Packings} \label{section:square-packings}

Our technique here using continued fractions is closely related to a result by Kari Nurmela, Patric \"Ostergard, and Rainer aus den Spring \cite{Nurmela} where it is applied to packing of equal circles in a square.  Their packings in a square are also such that the graph of the packing centers is along two sets of parallel lines, where the slope is given by a rational approximation to $1/\sqrt{3}$.  We show in Figure \ref{fig:30'} how their first example corresponds to our Type I packing with slope $5/3$ of $30$ disks in a square torus.
\begin{figure}[h]
    \begin{center}
        \includegraphics[width=0.3\textwidth]{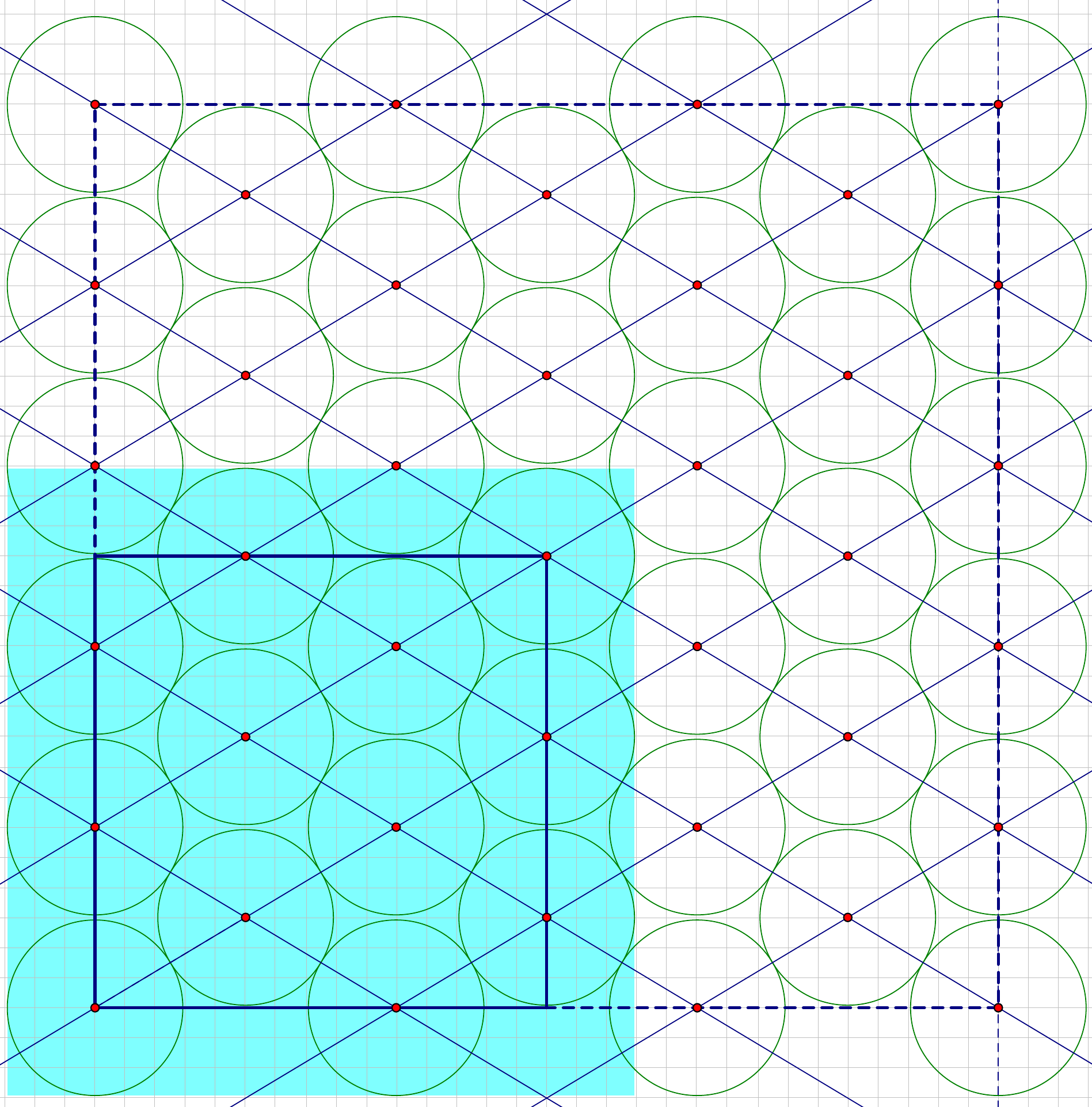}%
        \end{center}
    \caption{This shows the fundamental region of a packing of $30$ disks in a square torus, where one-fourth of that region, a smaller square half the length of one side of the fundamental region, contains $12$ centers out of the $30$.  When each side of that square is enlarged by diameter of the packing circles, we see the packing of $12$ disks in the blue square given by \cite{Nurmela}.}
    \label{fig:30'}
    \end{figure}

\section{Density bounds} \label{section:limiting-density}

\subsection{Unattainability of the limiting density}\label{subsection:no-limit}
\begin{theorem}[Thue (1892) \cite{Thue}] Any packing of equal disks in the plane has density $\delta \le \pi/\sqrt{12}$.
\end{theorem}
\begin{proof}
While this result was initially proven by Thue, we present here a proof given by Hai-Chau Chang and Lih-Chung Wang \cite{Chang-Wang}. 
\begin{definition}
A \emph{Delauney triangulation} of a set of points in the plane is one satisfying the property that no point in the set is contained in the interior of the circumcircle of any triangle.
\end{definition}

\begin{definition}
A \emph{saturated} configuration of circles of equal radius in a subset of $\R^2$ is one that is not a proper subset of another configuration. In other words, no circle can be added without overlapping some previous circle.
\end{definition}

\begin{lemma}
Delauney triangulations exist for saturated configurations.
\end{lemma}

To see this, assume that we have a set of points $\p_1, \dots, \p_n \in \R^2$ corresponding to the centers of a saturated circle configuration, with $\p_i = (x_i,y_i)$. Then lift the points $\p_1, \dots, \p_n$ to points $\q_1, \dots, \q_n \in \R^3$, with $\q_i = (x_i,y_i,x_i^2 + y_i^2)$, so that each point $\q_i$ lies on the paraboloid $z = x^2 + y^2$. Then taking the convex hull of the paraboloid gives a polyhedron with vertices at each $\q_i$. We project the convex hull down to the plane $\R^2$ to give a triangulation of the points. Since the faces of the convex hull arise from intersections with a given plane in $\R^3$, they correspond to circles in $\R^2$; in fact, these intersections correspond precisely to circumcircles of triangles. Because they originated from a complex hull, these circumcircles do not contain any other points in the configuration.

A slight hitch arises if these circles contain more than three points of the configuration; if the convex hull includes a square or an $n$-gon for $n \ge 4$, then that $n$-gon can be triangulated in any way to make a Delauney triangulation. The following lemma will show that for a saturated circle configuration, no hexagons or higher will arise, by showing that a Delauney triangulation of a saturated configuration cannot contain a triangle with any internal angle larger than $\frac{2\pi}3$.  Nevertheless, larger $n$-gons may be triangulated in any way without breaking the Delauney requirement. In a saturated configuration, \emph{any} of these triangulations must follow the angle requirement below.

\begin{lemma}
Let $\theta$ be the largest internal angle of a triangle $\bigtriangleup ABC$ in a Delauney triangulation for a saturated circle configuration of unit circles $\mathcal C$. Then $\frac\pi 3 \le \theta < \frac{2\pi}3$.
\end{lemma}
\begin{proof}
Since $\theta$ is the largest angle in a well-defined triangle, we immediately know that $\frac\pi 3 \le \theta$.

Assume by contradiction that $\theta \ge \frac{2\pi}3$. Let $\alpha$ be the smallest internal angle, at the vertex $A$ as depicted.
\begin{figure}{h}
\[\begin{tikzpicture}
\node (A) at (-0.5,0) {$A$};
\node (a) at (0.41,0.08) {$\alpha$};
\node (B) at (5.3, 3.1) {$B$};
\node (C) at (5.5,-1) {$C$};

\draw (0,0) -- (5,3) -- (5,-1) -- (0,0);
\end{tikzpicture}
\]
\caption{}
    \label{fig:triangle}
\end{figure}
Since the largest interior angle is $\ge \frac{2\pi}3$, $\alpha \le \frac{\pi}3$, so $\sin \alpha \le \frac 12$. Because we have a saturated circle configuration, $\overline{BC} \ge 2$. Let $r$ be the circumradius of $\bigtriangleup ABC$. The law of sines then tells us that
\[2r = \frac{\overline{BC}}{\sin \alpha} \ge \frac 2{\sin\alpha} \ge 4. \]

But since the circumradius is $\ge 4$, a circle can be added to the configuration at the circumcenter. Thus $\theta < \frac {2\pi}3$.
\end{proof}

Note that this is a strict inequality; the largest angle in a Delauney triangulation cannot be $\frac{2\pi}3$.

The density of a circle packing restricted to a specified triangle $\bigtriangleup ABC$ is 
\[\frac{\frac 12 A + \frac 12 B + \frac 12 C}{\text{area of }\bigtriangleup ABC} = \frac{\pi/2}{\text{area of }\bigtriangleup ABC}. \]

This allows for a computation of the density of a Delauney triangulation in general.

\begin{lemma}
The density $\delta$ of a triangle $\bigtriangleup ABC$ in a Delauney triangulation for a saturated circle configuration $\mathcal C$ satisfies $\delta \le \frac{\pi}{\sqrt{12}}$. Equality holds only with equilateral triangles of side length $2$.
\end{lemma}
\begin{proof}
Let $\beta$ be the largest internal angle, at $B$. Let $x$ be the area of $\bigtriangleup ABC$. Then 
\[x = \frac 12 \overline{AB} \overline{BC} \sin \beta \ge \frac 12 \cdot 2 \cdot 2 \cdot \frac{\sqrt{3}}2 = \sqrt{3}. \]

Thus the density of the triangle, which is given by $\frac{\pi/2}x$, is less than or equal to $\frac{\pi}{2\sqrt{3}} = \frac{\pi}{12}$.

In order for equality to hold, equality must hold for every part of the inequality $\frac 12 \overline{AB} \overline{BC} \sin \beta \ge \frac 12 \cdot 2 \cdot 2 \cdot \frac{\sqrt{3}}2$. In particular, $\beta = \frac{\pi}3$, and $\overline{AB} = \overline{BC} = 2$. Thus $\bigtriangleup ABC$ must be equilateral with side length $2$.
\end{proof}

For a finite union of Delauney triangles, we have a packing density $\delta$ with 

\[\delta = \frac{\sum_{\bigtriangleup_i} (x_i)\times (\delta_i)}{\sum_{\bigtriangleup_i}\delta_i},\]
where $x_i$ is the area of the triangle $\bigtriangleup_i$ and $\delta_i$ is the density of the triangle $\bigtriangleup_i$. As a result, the density of each Delauney triangle in the configuration is a bound on the total density, so to obtain the upper bound of $\frac{\pi}{\sqrt{12}}$, each triangle must have the limiting density. Thus the hexagonal lattice, with a Delauney triangulation consisting of the tiling of the plane with equilateral triangles, must be the unique packing of unit circles in finite subsets $\R^2$ achieving the density $\frac{\pi}{\sqrt{12}}$. We can extend this to periodic packings of $\R^2$ by noting that when restricted to any fundamental region, the packing must be the triangular lattice, so the entire packing must be the triangular lattice. In particular, this shows that $\frac{\pi}{\sqrt{12}}$ is an upper bound on the density of a periodic packing of unit circles in the plane, attained only by the hexagonal lattice packing.

\end{proof}

\begin{theorem}\label{thm:limit}
Any packing of a finite number of equal disks in a square torus has density $\delta < \pi/\sqrt{12}$.
\end{theorem}

\begin{proof}
Let the origin in the plane represent one of the packing centers in the torus.
If it lies on the triangular lattice, it can be represented by the complex number $n_1+m_1\omega$ where $\omega = \frac{1}{2}+i\frac{\sqrt{3}}{2}$ and $m_1,m_2$ are integers. 

Since the point $(0,1)$ lies on the triangular lattice, it must be able to be represented by the complex number $n_2+m_2\omega$ (with $n_2,m_2$ integers) and must be the product of $i$ and the lower right corner. 

Thus the equality $i(n_1+m_1\omega)=n_2+m_2\omega$ must be true if the square torus is to be represented by the triangular lattice.

Expanding, we have

\begin{align*}
i(n_1+m_1\omega)&=n_2+m_2\omega \\
\Rightarrow in_1+im_1\omega&=n_2+m_2\omega\\
\Rightarrow in_1+i\frac{1}{2}m_1-\frac{\sqrt{3}}{2}m_1 &= n_2+\frac{1}{2}m_2+i\frac{\sqrt{3}}{2}im_2 \\
i(n_1+\frac{1}{2}m_1-\frac{\sqrt{3}}{2}m_2) &= n_2+\frac{1}{2}m_2+\frac{\sqrt{3}}{2}m_1\\
\end{align*}

The left-hand-side has no real component, while the right-hand-side has no imaginary component, so this equality can only be satisfied if both sides sum to $0$.

Considering the left-hand-side, the equality $n_1+\frac{1}{2}m_1-\frac{\sqrt{3}}{2}m_2 = 0$ must be true.

Because $\frac{\sqrt{3}}{2}$ is irrational, no combination of integer multiples of $1, \frac{1}{2}, \frac{\sqrt{3}}{2}$, $(a_1+a_2\frac{1}{2}+a_3\frac{\sqrt{3}}{2}$) can sum to zero unless $a_1=a_2=a_3=0$.

These values of $a_1,a_2,a_3$ are only present when $n_1=m_1=n_2=m_2=0$, i.e. when the torus being represented is a single point at the origin.

Thus no square torus with nonzero area can be represented by points on the triangular lattice.

Because the only packing of disks with optimal density $\delta = \pi/\sqrt{12}$ is the packing represented by the triangular lattice, and this configuration can not be achieved in the square torus, any packing on the square torus must have density $\delta < \pi/\sqrt{12}$.
\end{proof}

\subsection{Gruber bounds for a planar square}\label{subsection:Gruber-bounds}

We state the following result of Norman Oler \cite{Oler} mentioned in the Introduction.

\begin{theorem}[Oler (1961)] Let $X$ be a compact polygon in the plane and $N$ the maximum number of distinct points in $X$ whose pairwise mutual distances are all at least $1$.  Then 
\[
N \le \frac{2}{\sqrt{3}}A(X) + \frac{1}{2}P(X) + 1, 
\]
where $A(X)$ is the area of $X$ and $P(X)$ is the perimeter of $X$.
\end{theorem}

\begin{theorem}\label{thm:best-square}
Assume we have any packing of $N$ equal circles in a square with density $\delta$. Then there exists a constant $k>0$ such that \[\left|\frac{\pi}{\sqrt{12}} - \delta \right| > \frac{k}{\sqrt{N}}. \] 
\end{theorem}
\begin{proof}
Note that $\frac{\pi}{\sqrt{12}} - \delta > 0$, because $\frac{\pi}{\sqrt{12}}$ is an upper bound on the packing density of tilings of the plane and thus on the packing density of tilings of the square. We will write $\frac{\pi}{\sqrt{12}} - \delta$ instead of $|\frac{\pi}{\sqrt{12}} - \delta|$.

We may assume that the radius of our equal circles is $1$, and vary the width of the square. Oler's theorem tells us that for a fixed side length $a$ of a square in $\EE^2$, the number $N$ of points that may be placed within the square, all of which are at distance at least $1$, satisfies 
\[N \le \frac 2{\sqrt{3}} a^2 + 2a + 1. \]
Consider our square containing circles as containing centered within it a smaller square, with boundary margin of width $1/2$. Let $a$ be the side length of the smaller square, so that $a+1$ is the side length of the square in the packing. All centers of circles must lie within the smaller square, so Oler's theorem applies; thus for a square of side length $a+1$, we have $N \le \frac 2{\sqrt{3}} a^2 + 2a + 1.$

The density $\delta$ is given by 
\begin{align*}
\delta &= \frac{\pi(\frac 12)^2N}{(a+1)^2} \\
\Rightarrow \delta &\le \frac{\pi\left(\frac 2{\sqrt{3}} a^2 + 2a + 1\right)}{4(a+1)^2}.
\end{align*}

This leads to the computation that

\begin{align*}
\frac{\pi}{\sqrt{12}} - \delta &\ge \frac{\pi}{2\sqrt{3}} - \frac{\pi}{4(a+1)^2}\left(\frac 2{\sqrt{3}} a^2 + 2a + 1 \right) \\
\Rightarrow \frac{\pi}{\sqrt{12}} - \delta &\ge \frac \pi 4 \left(\frac 2{\sqrt{3}} -  \frac 2{\sqrt{3}}\frac{a^2}{(a+1)^2} - \frac{2a}{(a+1)^2} - \frac 1{(a+1)^2} \right) \\
\Rightarrow \frac{\pi}{\sqrt{12}} - \delta &\ge \frac \pi 4 \left(\frac 2{\sqrt{3}} - \frac 2{\sqrt{3}} \frac{a^2 + 2a + 1}{(a+1)^2} - \left(1 - \frac 2{\sqrt{3}} \right) \left(\frac{2a+1}{(a+1)^2} \right) \right) \\
\Rightarrow \frac{\pi}{\sqrt{12}} - \delta &\ge \frac \pi 4 \left(\frac 2{\sqrt{3}} - 1\right) \left(\frac{2a + 1}{(a+1)^2}\right). \\
\end{align*}

We examine two cases: when $\frac{\pi}{\sqrt{12}} - \delta < \frac 12$, and when $\frac{\pi}{\sqrt{12}} - \delta \ge \frac 12$.

\paragraph{Case 1: $\frac{\pi}{\sqrt{12}} - \delta \ge \frac 12$:}

Since we have $N \ge 1$, $N^{1/2} \ge 1$, so $(\frac{\pi}{\sqrt{12}} - \delta) N^{1/2} \ge \frac 12$. Thus there exists a constant $k_1 = \frac 12$ so that $(\frac{\pi}{\sqrt{12}} - \delta) \ge \frac{k_1}{N^{1/2}}$.

\paragraph{Case 2: $\frac{\pi}{\sqrt{12}} - \delta < \frac 12$:}

In this case, $\frac{\pi}{\sqrt{12}} - \frac 12 < \frac \pi 4 \frac N{(a+1)^2}.$ We can then compute:

\begin{align*}
N &> \frac{4(a+1)^2}{\pi} \left(\frac{\pi}{2\sqrt 3} - \frac 12 \right) \\
\Rightarrow N^{1/2} &> \frac{2(a+1)}{\sqrt\pi} \sqrt{\frac{\pi}{2\sqrt 3} - \frac 12} \\
\Rightarrow \left(\frac{\pi}{\sqrt{12}} - \delta \right) N^{1/2} &> \frac \pi 4 \left(\frac 2{\sqrt{3}} - 1\right) \left(\frac{2a + 1}{(a+1)^2}\right)\frac{2(a+1)}{\sqrt\pi} \sqrt{\frac{\pi}{2\sqrt 3} - \frac 12} \text{, using the bound from above} \\
\Rightarrow \left(\frac{\pi}{\sqrt{12}} - \delta \right) N^{1/2} &> \frac{\sqrt \pi}{2} \left(\frac 2{\sqrt 3} - 1 \right) \frac{2a + 1}{a + 1}\sqrt{\frac{\pi}{2\sqrt 3} - \frac 12} \text{, where $a > 1$} \\
\Rightarrow \left(\frac{\pi}{\sqrt{12}} - \delta \right) N^{1/2} &> \frac{\pi}2 \left(\frac 2{\sqrt{3}} - 1 \right) \frac 32 \sqrt{\frac\pi{\sqrt{12}} - \frac 12}.
\end{align*}

Then setting $k_2 = \frac{\pi}2 \left(\frac 2{\sqrt{3}} - 1 \right) \frac 32 \sqrt{\frac\pi{\sqrt{12}} - \frac 12} \approx 0.2325$, we have $\left(\frac{\pi}{\sqrt{12}} - \delta \right) N^{1/2} \ge k_2$.

Since $k_2 < k_1$, we always have $\left(\frac{\pi}{\sqrt{12}} - \delta \right) N^{1/2} \ge k_2$, just as desired.
\end{proof}

\section{Numerical Experiments and Proposed Best Packings} \label{section:Best}
 In order to experimentally determine optimal densities, radii, and configurations for certain numbers of disks, denoted $N$, on a square torus, we wrote a program in Java using chaos to find the most tightly-packed configurations following a similar program by Alex Smith for the triangular torus. The program was most reliable for fairly small values, i.e. $N < 30$, as it consistently converged to the same configuration and found no arrangement with a higher density.

For $1 \le N \le 5$, optimal packings and their corresponding densities have been known for some time as mentioned above, and for $6 \le N \le 8$ the highest density configurations were reported by Musin and Nikitenko \cite{Musin}. They also offered a conjecture for $N = 9$. Presented in the following figures are the configurations, radii, and densities and Markov constant for $N = 6$ to $16$.  Here it is assumed that the total area is $1$. The findings for $6$ to $9$ disks corroborate the findings and conjectures published by \cite{Musin}, and the findings for $10$ to $16$ disks we offer as conjectures.
\newpage
\subsection{$6 \le N \le 9$: Corroborating Musin and Nikitenko}\label{subsection:Musin}
\begin{figure}[!htb]
    \centering
    \begin{minipage}{.5\textwidth}
        \centering
        \includegraphics[width=0.7\linewidth]{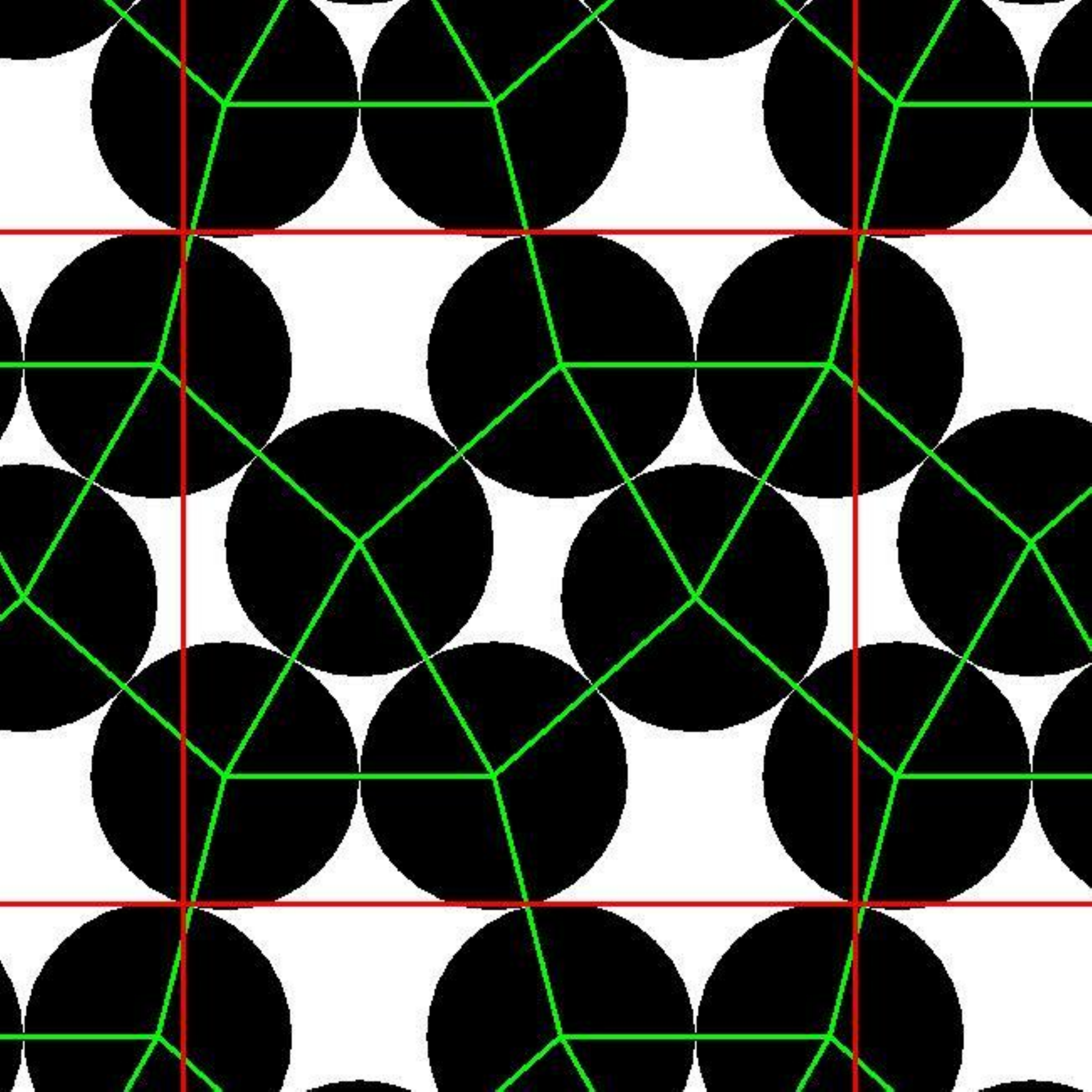}
        \caption{$N=6$}
        \label{fig:6}
    \end{minipage}%
    \begin{minipage}{0.5\textwidth}
        \centering
        \includegraphics[width=0.7\linewidth]{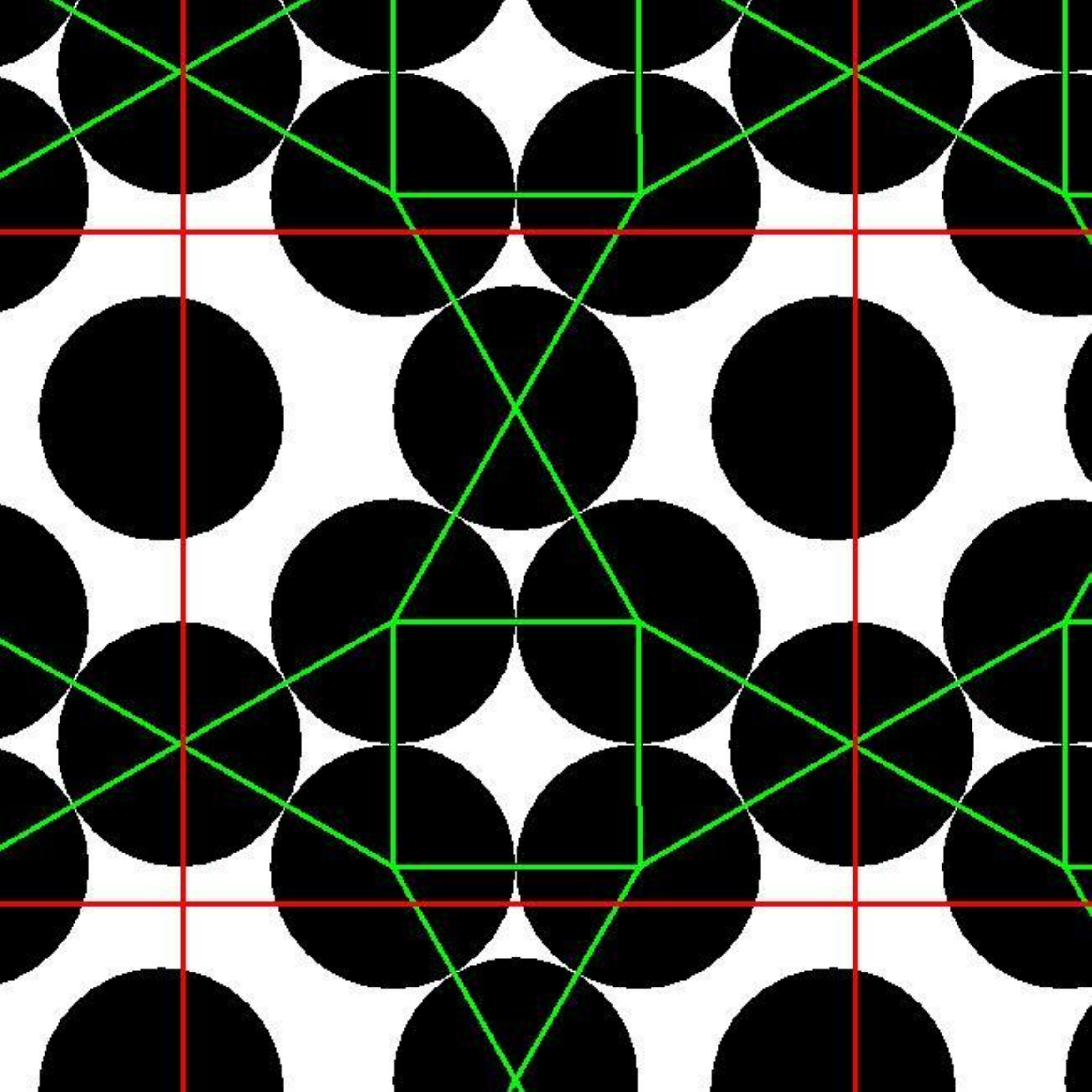}
        \caption{$N=7-1$}
        \label{fig:7-1}
         \end{minipage}\\%
     \begin{minipage}{.5\textwidth}
        \centering
        \includegraphics[width=0.7\linewidth]{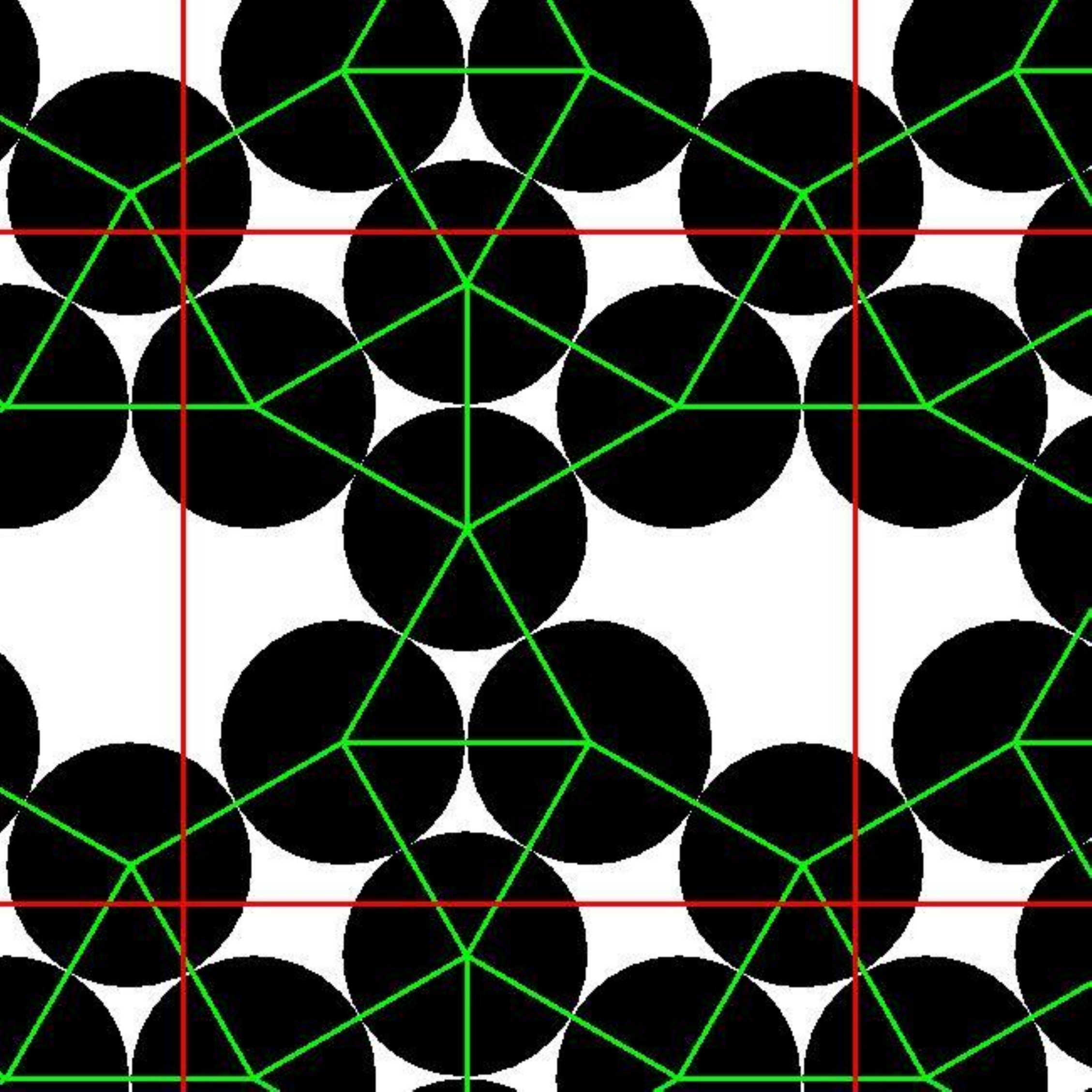}
        \caption{$N=7-2$}
        \label{fig:7-2}
    \end{minipage}%
    \begin{minipage}{0.5\textwidth}
        \centering
        \includegraphics[width=0.7\linewidth]{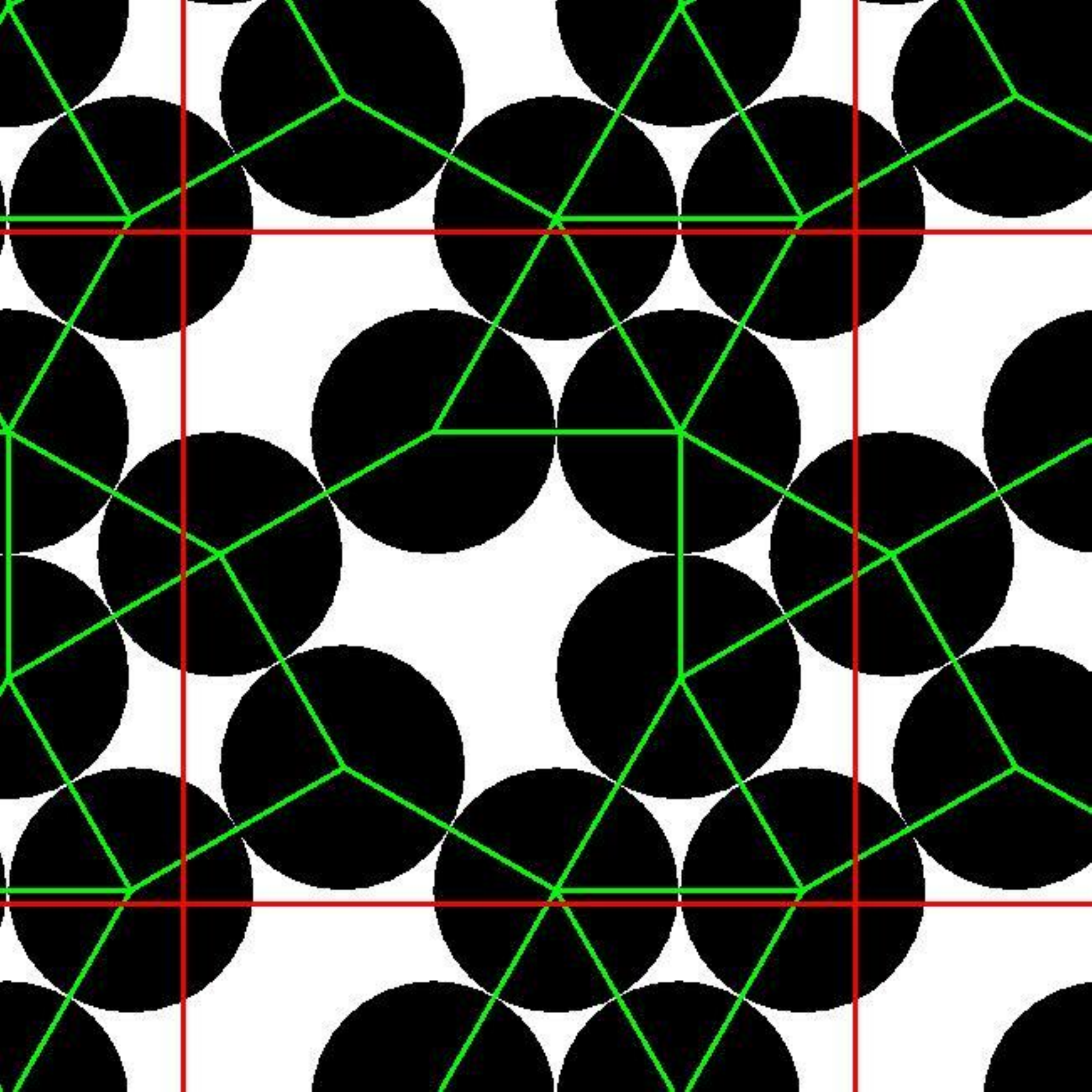}
        \caption{$N=7-3$}
        \label{fig:7-3}    
        \end{minipage}\\%
      \begin{minipage}{.5\textwidth}
        \centering
        \includegraphics[width=0.7\linewidth]{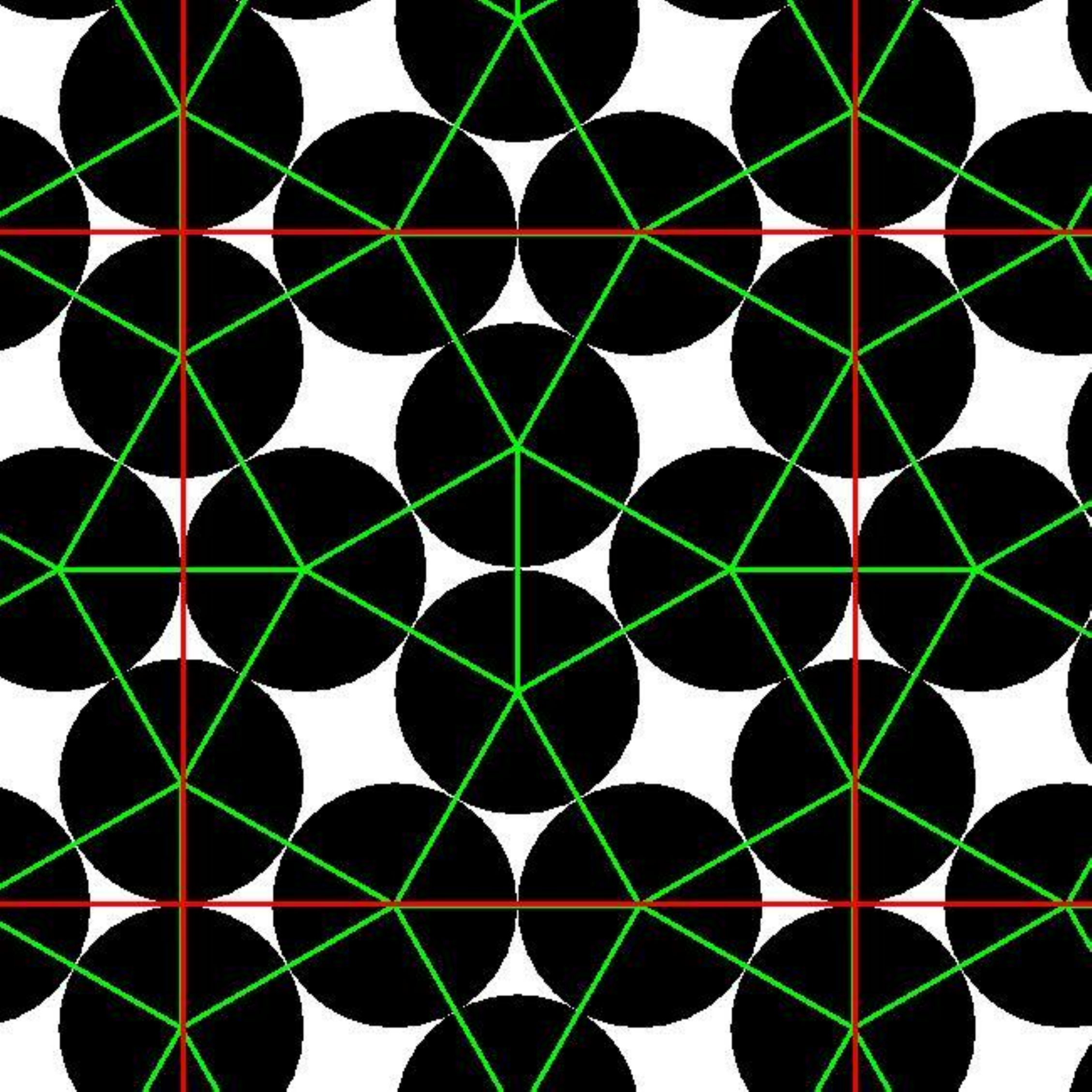}
        \caption{$N=8$}
        \label{fig:8'}
    \end{minipage}%
    \begin{minipage}{0.5\textwidth}
        \centering
        \includegraphics[width=0.7\linewidth]{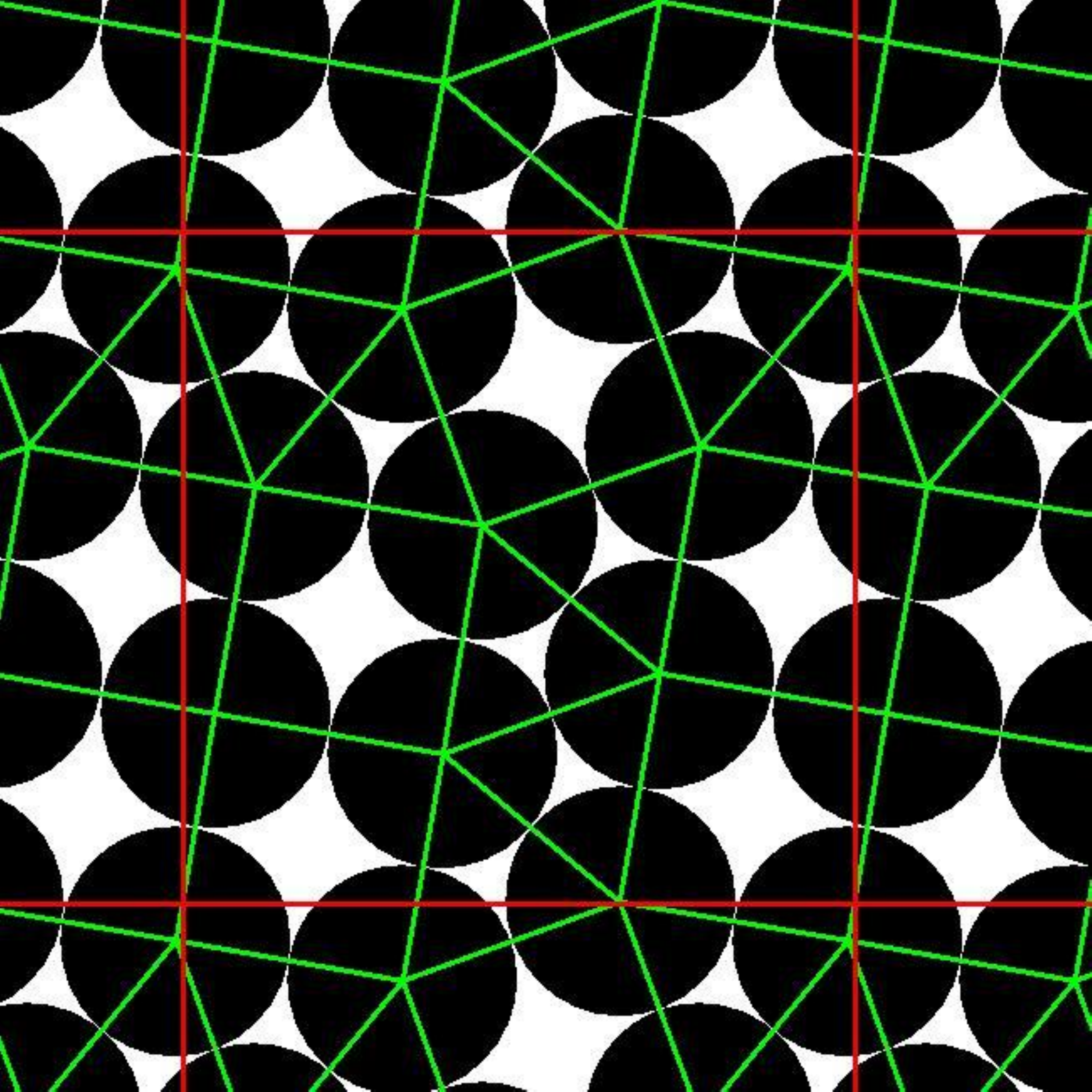}
        \caption{$N=9$}
        \label{fig:9}    
        \end{minipage}\\%
    \end{figure}
\clearpage
\subsection{$10 \le N \le 16$: Conjectures}\label{subsection:conjectures}
\begin{figure}[!htb]
    \centering
    \begin{minipage}{.5\textwidth}
        \centering
        \includegraphics[width=0.7\linewidth]{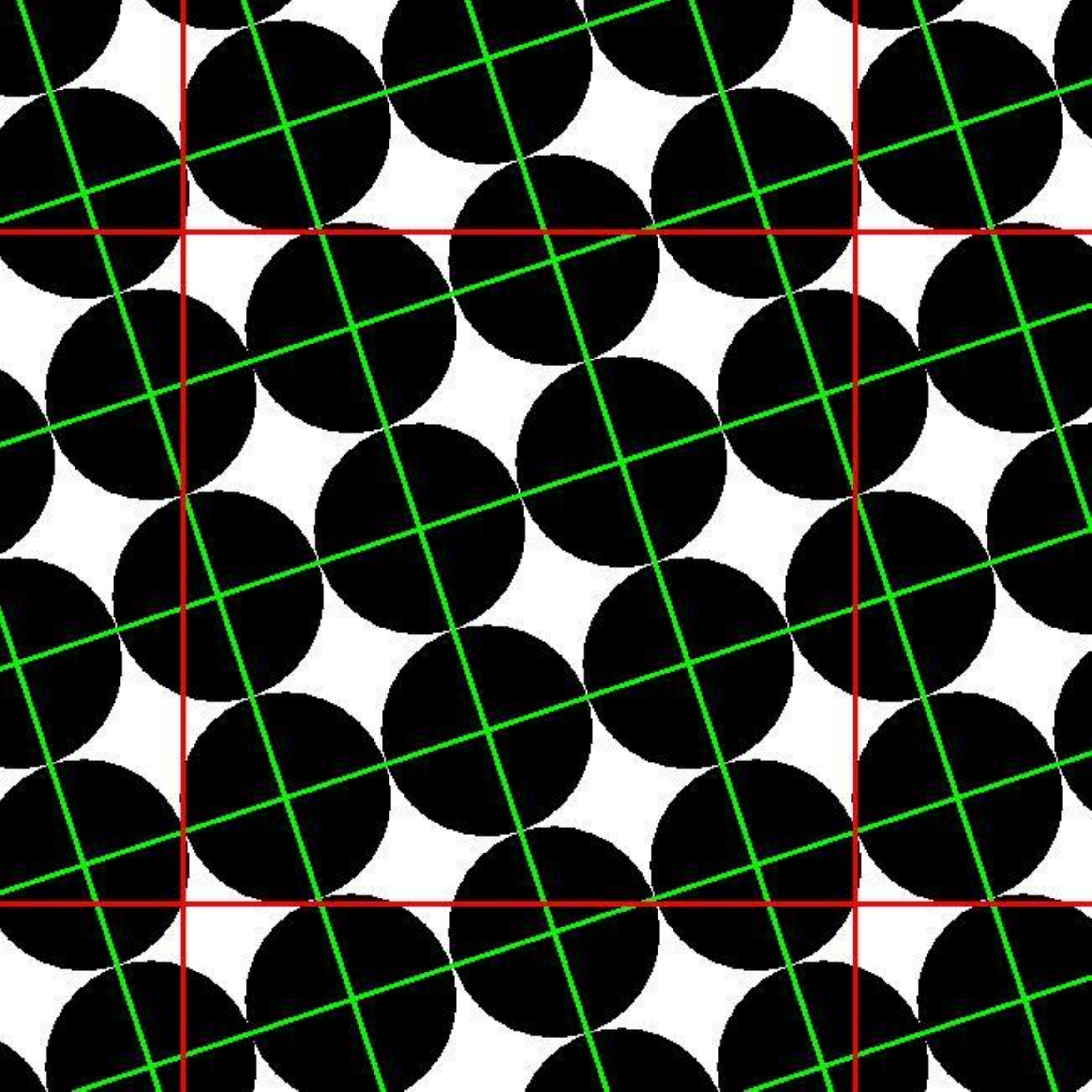}
        \caption{$N=10$}
        \label{fig:10}
    \end{minipage}%
    \begin{minipage}{0.5\textwidth}
        \centering
        \includegraphics[width=0.7\linewidth]{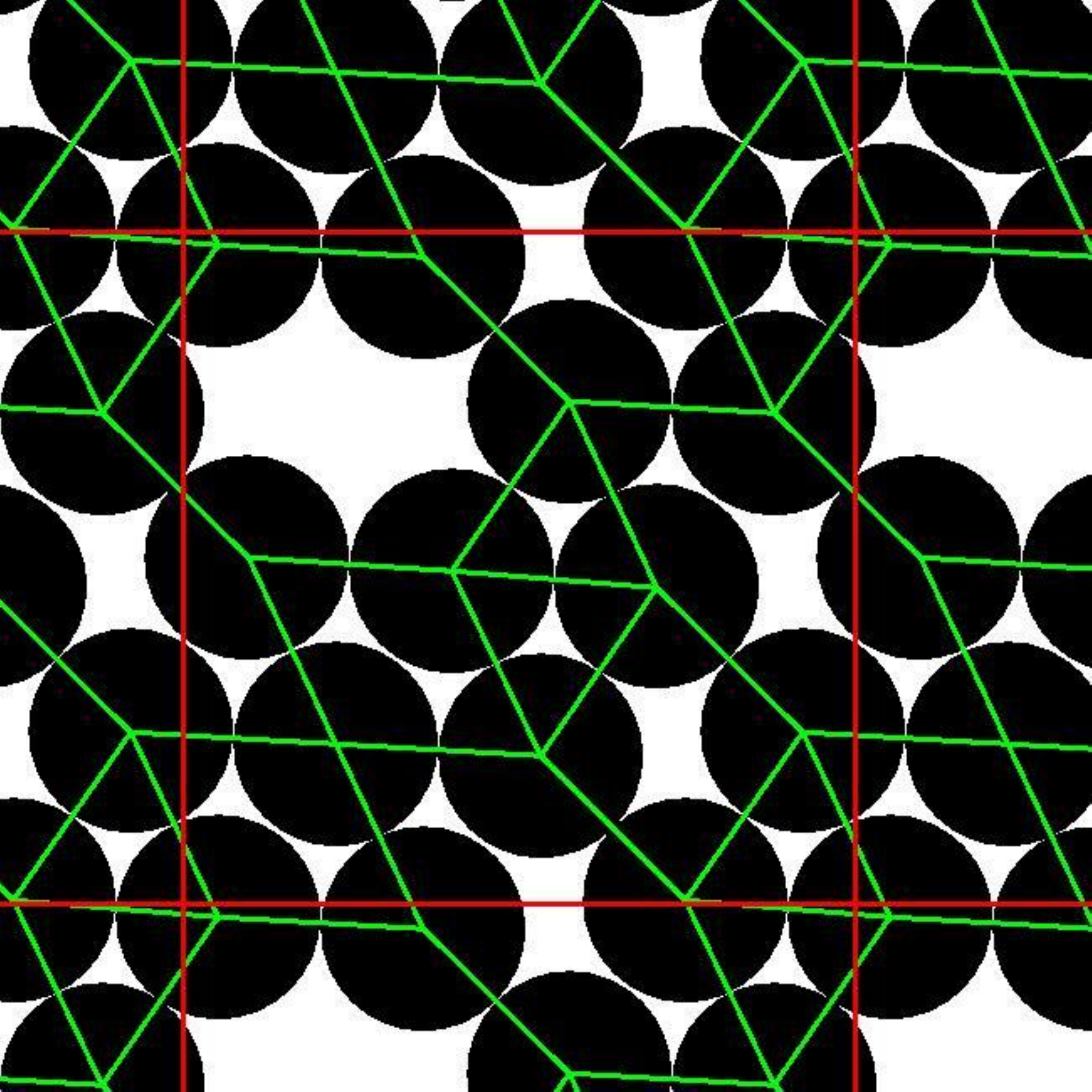}
        \caption{$N=11$}
        \label{fig:11}
         \end{minipage}\\%
     \begin{minipage}{.5\textwidth}
        \centering
        \includegraphics[width=0.7\linewidth]{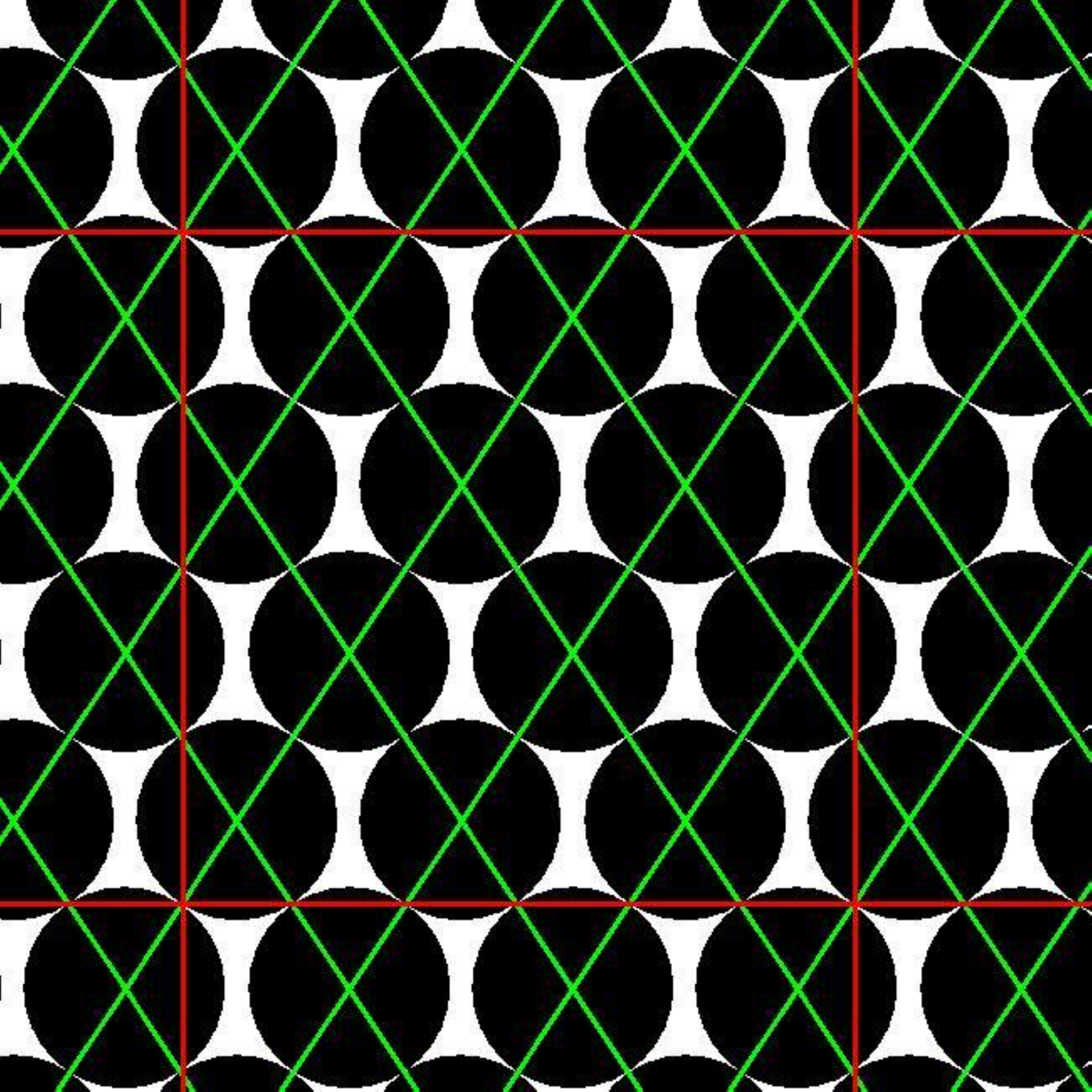}
        \caption{$N=12$}
        \label{fig:12'}
    \end{minipage}%
    \begin{minipage}{0.5\textwidth}
        \centering
        \includegraphics[width=0.7\linewidth]{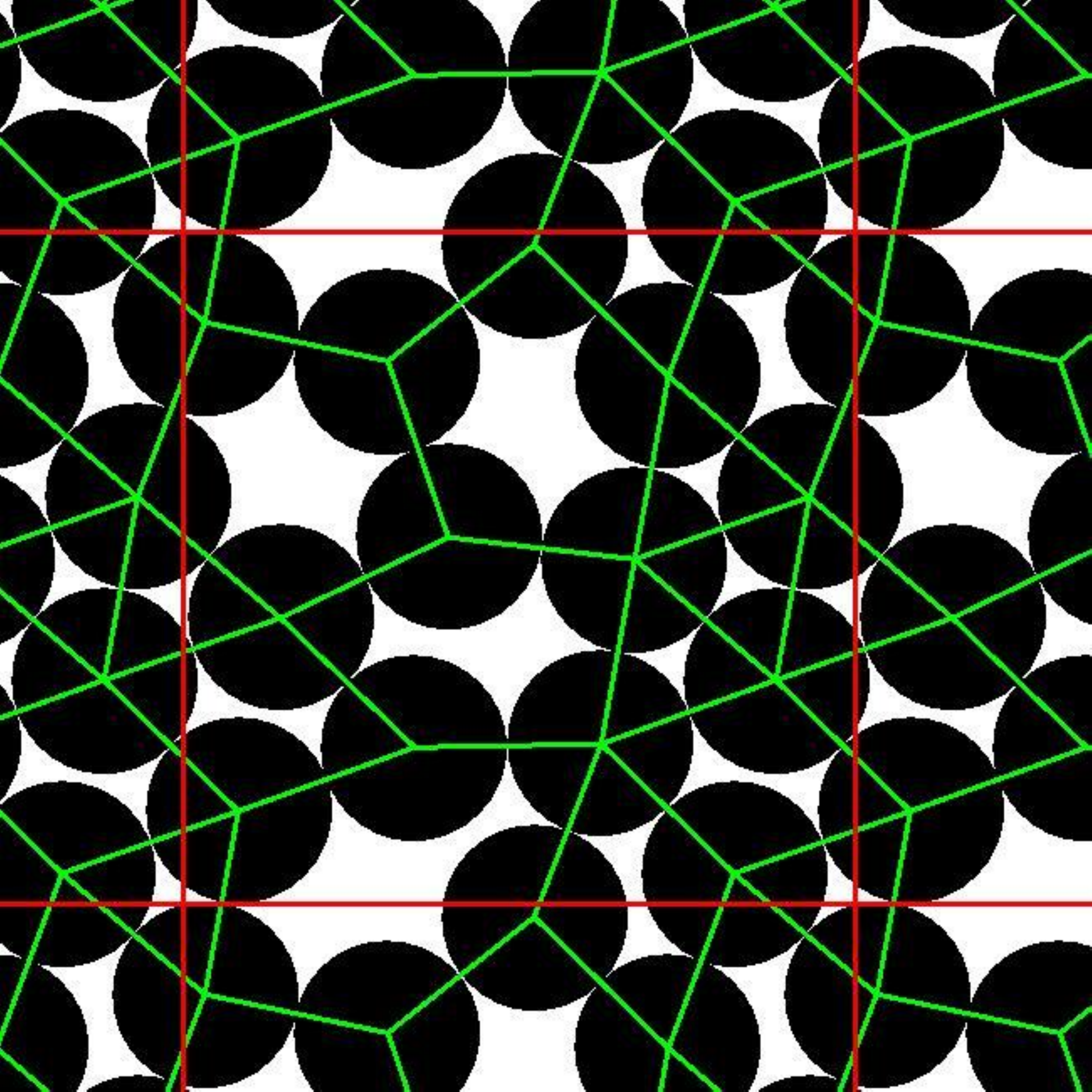}
        \caption{$N=13$}
        \label{fig:13}    
        \end{minipage}\\%
      \begin{minipage}{.5\textwidth}
        \centering
        \includegraphics[width=0.7\linewidth]{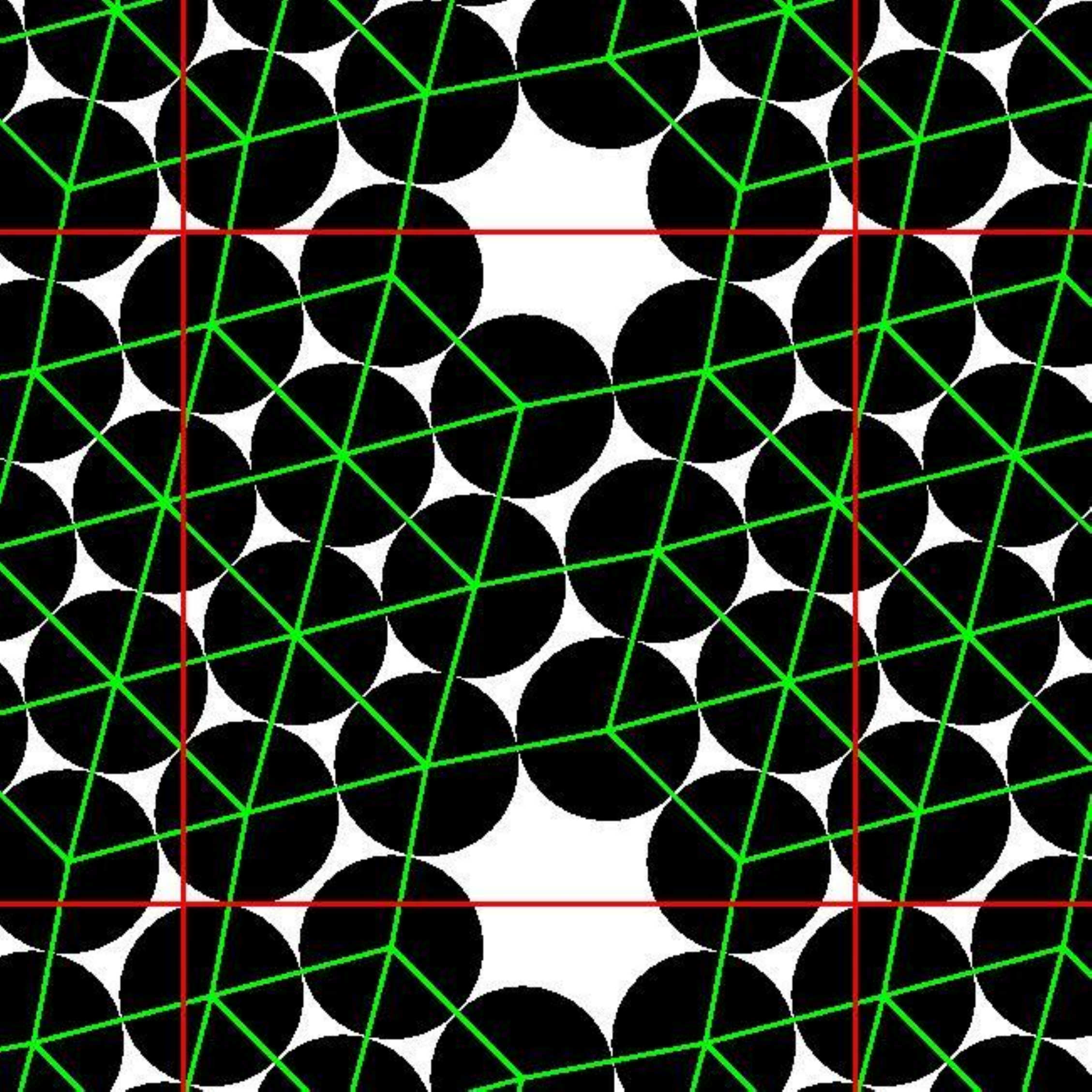}
        \caption{$N=14$}
        \label{fig:14}
    \end{minipage}%
    \begin{minipage}{0.5\textwidth}
        \centering
        \includegraphics[width=0.7\linewidth]{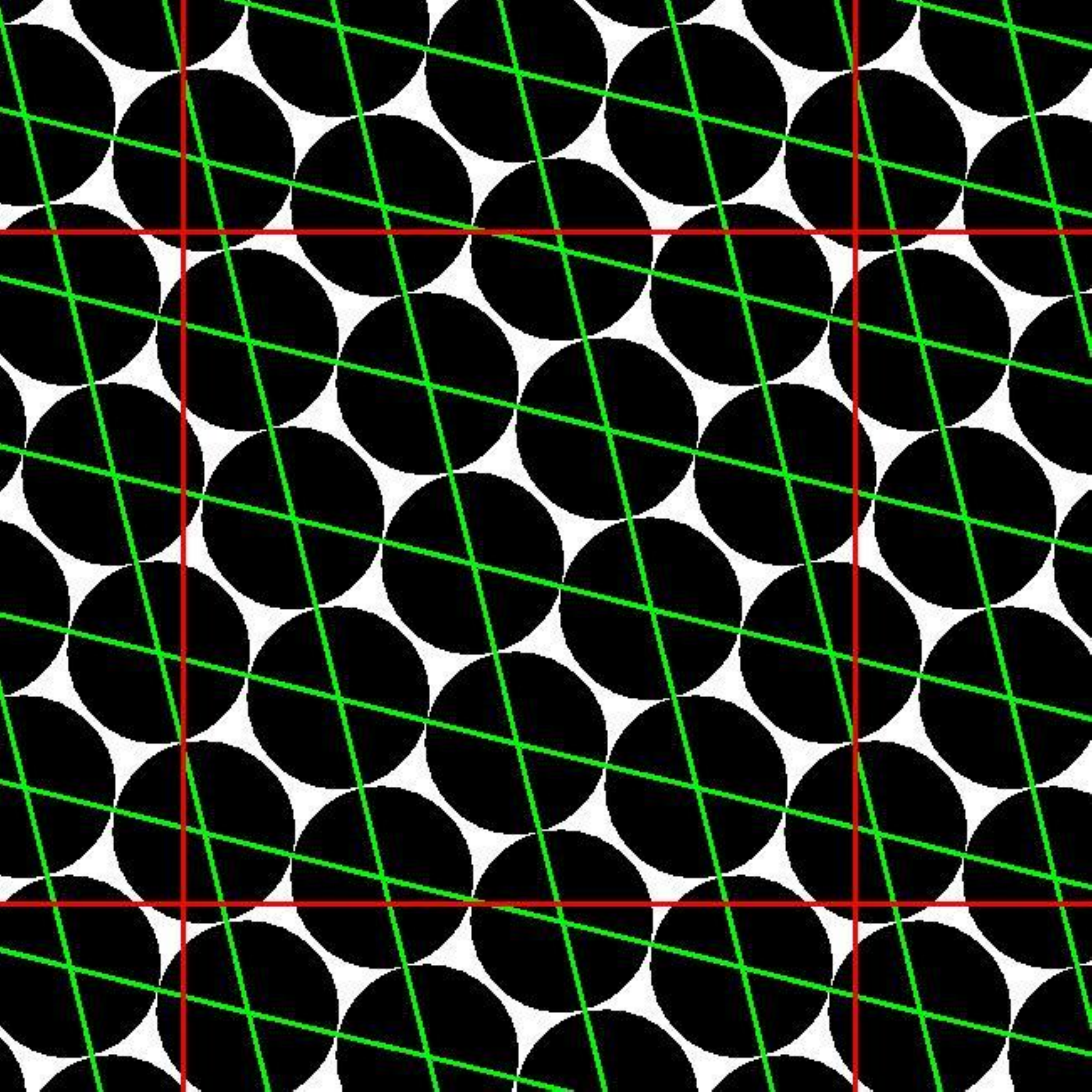}
        \caption{$N=15$}
        \label{fig:15}    
        \end{minipage}\\%
    \end{figure}
    \clearpage
\begin{figure}[!htb]
    \centering
    \begin{minipage}{.5\textwidth}
        \centering
        \includegraphics[width=0.7\linewidth]{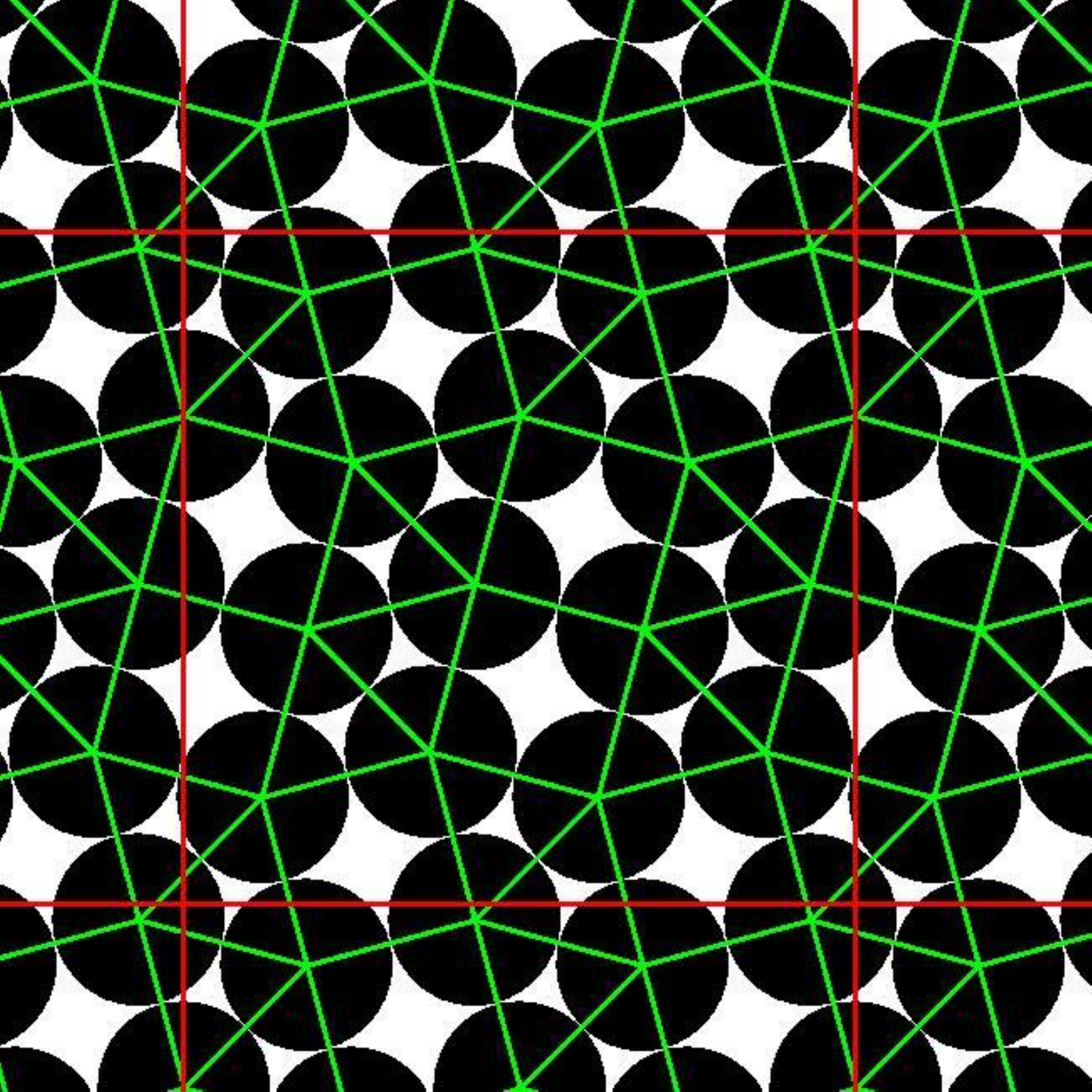}
        \caption{$N=16-1$}
        \label{fig:16-1}
    \end{minipage}%
    \begin{minipage}{0.5\textwidth}
        \centering
        \includegraphics[width=0.7\linewidth]{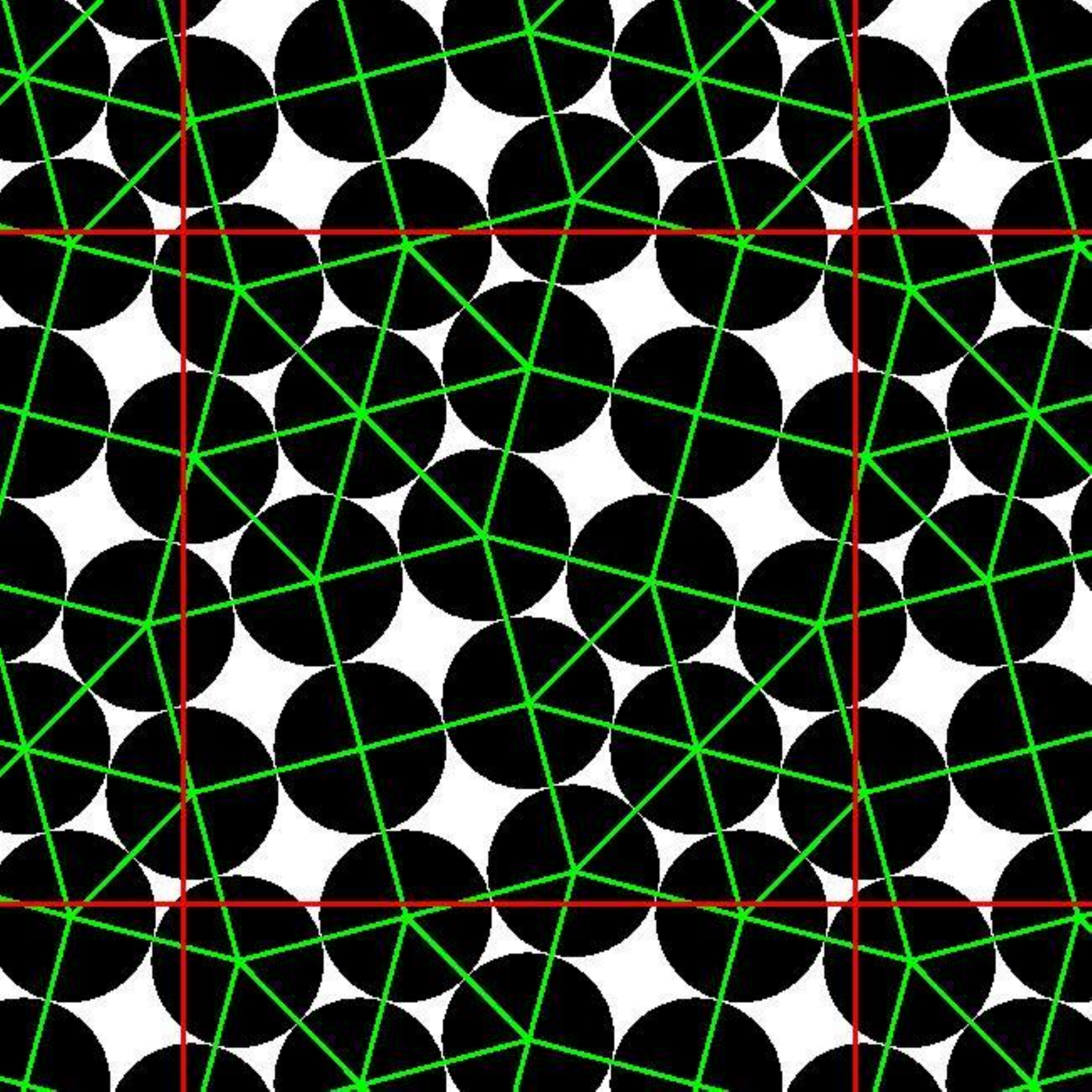}
        \caption{$N=16-2$}
        \label{fig:16-2}
         \end{minipage}
\end{figure}
\begin{center}
\begin{tabular}{|c|c|c|c|}
\hline
$N$ (disks) & $\delta$ (density) & $r$ (radius) & M(Markov)\\
\hline
$1$ & $0.785$ & $0.500$ &12.93\\
$2$ & $0.785$ & $0.354$ &6.464\\
$3$ & $0.631$ & $0.259$ &1.900\\
$4$ & $0.842$ & $0.259$ &6.031\\
$5$ & $0.785$ & $0.224$ &2.586\\
$6$ & $0.755$ & $0.200$ &1.712\\
$7$ & $0.736$ & $0.183$ &1.316\\
$8$ & $0.842$ & $0.183$ &3.010\\
$9$ & $0.835$ & $0.172$ &2.478\\
$10$ & $0.785$ & $0.158$ &1.280\\
$11$ & $0.806$ & $0.153$ &1.457\\
$12$ & $0.851$ & $0.150$ &2.231\\
$13$ & $0.788$ & $0.139$ &1.025\\
$14$ & $0.839$ & $0.138$ &1.619\\
$15$ & $0.890$ & $0.137$ &4.668\\
$16$ & $0.842$ & $0.129$ &1.393\\
\hline
\end{tabular}
\end{center}

Notice that the calculation for the Markov constant here for $N=15$ is $4.668$, which is somewhat less than than $6.24$ which is the exact constant as shown in Figure \ref{fig:Dense-plot}.  The accuracy of the numerical technique decreases as the number $N$ of the packing disks increases.  

Taking these calculations as evidence, we make the following statements as conjectures.  Note that the Markov constants for Type I and Type II packing converge monotonically,  decreasing to their limits of $3$ and $6$ respectively, so the first few examples provide the upper bounds.

\begin{conjecture} For a packing of $N \ge 6$ disks in a square torus, the Markov constant 
$
M \le 6.25
$.
In other words if the density of a packing of $N\ge 6$ disks is $\delta$, then
\[
 \frac{\pi}{2\sqrt{3}}-\delta > \frac{\pi}{2}\left(\frac{1}{6.25 N}\right).
\]
\end{conjecture}

Notice that the corresponding statement for the planar square is true as mentioned in the introduction from \cite{Gruber-optimal}. 

The first part of the next conjecture essentially concerns number theory in that it has to do with what numbers can be written as the sum of two squares in two ways.  See the remark in Figure \ref{fig:Dense-plot}.  The other rigid packings with the centers on two geodesics can be obtained by regarding the generating vectors as Gausian integers in the complex plane obtained by multiplying the Type I (or Type II) generators by another Gaussian integer.

 \begin{conjecture} If a rigid packing of $N \ge 6$ disks in a square torus is such that its graph consists of two  linear geodesics, and its Markov constant $M > 3$, then it is a Type I or a Type II packing, and it is maximally dense.
\end{conjecture}

A dynamical phenomenon, is that as the disks in the torus grow and collide, there seems to be a point where they stop moving long distances and only vibrate toward a stable rigid configuration.  The configuration seems to be frozen, but not rigid.  This leads to another conjecture.

 \begin{conjecture} For any $N=2, 3, \dots$, there is an $\epsilon(N) \rightarrow 0$, such that for any packing of $N$ disks in a square torus with Markov constant $M > 3$, the packing is within $\epsilon(N)$ of a most most dense packing. 
\end{conjecture}
The distance between configurations can be taken with respect to any reasonable metric on the space of configurations, for example the Euclidean metric using all the coordinates of a configuration.

\bibliographystyle{plain}
\bibliography{framework}

\end{document}